\documentclass[a4paper, 11pt]{article}
\usepackage{amsmath}
\numberwithin{equation}{section}
\usepackage{amssymb,esint,hyperref}
\usepackage{amscd}
\usepackage{xspace}
\usepackage{fancyhdr}
\usepackage{color}
\usepackage{verbatim}
\usepackage{cite}
\usepackage{stmaryrd}
\usepackage{bbm}
\usepackage{mathrsfs}
\usepackage{srcltx} 
\usepackage{marginnote}
\let\oldmarginpar\marginpar
\renewcommand\marginpar[1]{\-\oldmarginpar[\raggedleft\footnotesize #1]%
	{\raggedright\footnotesize\color{red} #1}} 
\newtheorem{theorem}{Theorem}[section]

\newtheorem{lemma}[theorem]{Lemma}

\newtheorem{remark}[theorem]{Remark}

\newcommand{\eps}{\varepsilon}

\newtheorem{thm}{Theorem}[section]

\newenvironment{proof}[1][Proof]{\textbf{#1.} }{\hfill\rule{0.5em}{0.5em}}
{\catcode`\@=11\global\let\AddToReset=\@addtoreset
	\AddToReset{equation}{section}

	\AddToReset{theorem}{section}

	\title{Global well-posedness of the 1d compressible Navier–Stokes system with rough data}
	\author{
		{\bf Ke Chen,\thanks{E-mail address: kchen18@fudan.edu.cn, Fudan University, 220 Handan Road, Yangpu, Shanghai, 200433, China.}~~ Ly Kim Ha,\thanks{E-mail address: lkha@hcmus.edu.vn, Address 1: University of Science VNU-HCMC, Ho Chi Minh City 700000, Vietnam, Address 2: Vietnam National University, Ho Chi Minh City, Vietnam}
			~~Ruilin Hu,\thanks {E-mail address: huruilin16@mails.ucas.ac.cn,  Academy of Mathematics and Systems Science, Chinese Academy of Sciences, Beijing, 100190, China.}
			~~Quoc-Hung Nguyen\thanks{E-mail address: qhnguyen@amss.ac.cn, Academy of Mathematics and Systems Science, Chinese Academy of Sciences, Beijing, 100190, China.}}}}
\begin{document}
	\maketitle
	\begin{abstract} 
	In this paper, we study the global well-posedness problem for the 1d compressible Navier–Stokes systems (cNSE) in gas dynamics with rough initial data. First, Liu and Yu (Commun Pure Appl Math 75(2):223–348, 2022) established the global well-posedness theory for the 1d isentropic cNSE with initial velocity data in BV space. Then, it was extended to the 1d cNSE for the polytropic ideal gas with initial velocity and temperature data in BV space by Wang-Yu-Zhang (Arch Rational Mech Anal 245, 375–477,2022). We improve the global well-posedness result of Liu and Yu with initial velocity data in  $W^{2\gamma,1}$ space; and of Wang-Yu-Zhang with initial velocity data in $ L^2\cap W^{2\gamma,1}$ space and initial data of temperature in $\dot W^{-\frac{2}{3},\frac{6}{5}}\cap \dot W^{2\gamma-1,1}$ for any $\gamma>0$ \textit{arbitrary small}. Our essential ideas are based on establishing various ``end-point" smoothing estimates for the 1d parabolic equation. 
	\end{abstract}
	
	\section{Introduction}
	In this paper, we consider the compressible Navier--Stokes equations in Lagrangian coordinates, which can be written as (see \cite{Smoller}) 
	\begin{equation}\label{eqcompre}
		\left\{
		\begin{aligned}
			&	v_t-u_x=0,\\
			&	u_t+p_x=\left(\frac{\mu u_x}{v}\right)_x,\\
			&	(e+\frac{1}{2}u^2)_t+({p}u)_x=\left(\frac{\kappa}{v}\theta_x+\frac{\mu}{v}uu_x\right)_x.
		\end{aligned}
		\right.
	\end{equation}
	Here we denote $v$ the specific volume, $u$ the velocity, $p$ the pressure, $e$ the specific internal energy,  $\theta $ the temperature. And $\mu, \kappa>0$ are   viscosity and heat conductivity coefficients. The above equations describe the conservation of mass, momentum and energy, respectively. The system is hyperbolic-parabolic. It is not uniformly parabolic, but dissipative. Solutions will be singular at finite time with some large initial data. The purpose of this paper is to study the global well-posedness and long time behavior for  the system with initial data closing some constant states. 
	
	There are extensive literatures on the mathematical analyses of cNSE.  The study of cNSE in gas dynamics started by the pioneer work of Nash in \cite{Nash1958}, where he studied the continuity of general elliptic and parabolic equations. Then Itaya \cite{Itaya71,Itaya75}, Tani \cite{Tani}, Valli \cite{valli} proved local well-posedness in the absence of vacuum. The global classical solution was first obtained by Kazhikhov--Shelukhin \cite{Ka77}   for initial boundary value problem with $H^1$ data. The corresponding result for the Cauchy problem can be found in \cite{KaCau}. The global well-posedness of weak solutions was proved by Jiang--Zlotnik \cite{JiangZ}. The asymptotic behavior as $t\to\infty$ of the solution has been studied under
some smallness conditions on the initial data, see \cite{Hoff92,Liu2}. We also refer to \cite{Jiang99,LiLiang} for large time behavior of solutions with large initial data. 
For the the multi-dimensional case, Matsumura and Nishida \cite{MN1980} applied the energy method to derive the existence of a global solution with Sobolev data. The result of large-time behavior of global solution was provided by Kawashima \cite{Kawa87}. Using the energy method, Hoff \cite{Hoff92,Hoff95} proved the existence of a global solution with discontinuous data.

 In the presence of vacuum, Lions \cite{Lions} proved
the global existence of weak solutions to the isentropic cNSE, see also Feireisl \cite{Fe} from result of the full cNSE. Further developments can be found in \cite{Fe2,Jiang}. The global well-posedness of classical solutions was first proved by Huang--Li--Xin\cite{HLX}, their results hold for 3d isentropic system with small initial energy, see also \cite{LiXin2}. For the full cNSE with vacuum, Xin \cite{Xinbl} and Xin–Yan \cite{Xinbl2} proved finite time blow up results for classical solutions. 
We refer interested readers to  \cite{Jiang,LiWX,LWX,LiXin,LiXin2} for more related results.\\\\
	We will consider the following two systems. The first system is the  model of isentropic gas dynamics where $$p(v)=Av^{-\nu},\ A>0,\ 1\leq \nu<\mathrm{e}.$$ In the isentropic case, the temperature is held constant hence energy must be added to the system, whence the conservation 
	of energy is absent. The system can be written as  
	\begin{align}\label{inns}
		\left\{
		\begin{aligned}
			&	v_t-u_x=0,\\
			&	u_t+(p(v))_x=(\frac{\mu u_x}{v})_x,~~ p(v)=Av^{-\nu}, 
			\\
			& (v(0),u(0))=(v_0,u_0).
		\end{aligned}
		\right.
	\end{align}
The second system is \eqref{eqcompre} for the polytropic ideal gas which has the constitutive relations
	$$
	P(v,\theta)=\frac{K\theta}{v},\ \ \ e=\mathbf{c}\theta,
	$$
	where $K$ and $\mathbf{c}$ are both positive constants. The system can be written as 
	\begin{equation}\label{cpns}
		\left\{
		\begin{aligned}
			&	v_t-u_x=0,\\
			&	u_t+(P(v,\theta))_x=\left(\frac{\mu u_x}{v}\right)_x,\\
			&	\theta_t+\frac{P(v,\theta)}{\mathbf{c}}u_x-\frac{\mu}{\mathbf{c}v}(u_x)^2=\left(\frac{\kappa}{\mathbf{c}v}\theta_x\right)_x,\\
			& (v(0),u(0),\theta(0))=(v_0,u_0,\theta_0).
		\end{aligned}
		\right.
	\end{equation}
	Recently, Liu--Yu \cite{LiuCAPM} proved global well-posedness of weak solutions for  system \eqref{inns} with $BV\cap L^1$ data. More precisely, they proved that for initial data $(v_0,u_0)$ satisfying 
\begin{align}\label{theircon}
	\|(v_0-1,u_0)\|_{L^1\cap BV}\leq \delta,
\end{align}
	for sufficiently small $\delta$, the solution
exists global-in-time and, for some positive constant $C$,
\begin{align*}
&\|(v(t)-1,u(t))\|_{L^1\cap BV}+\sqrt{t+1}\|(v(t)-1,u(t))\|_{L^\infty}+\sqrt{t}\|u_x(t)\|_{L^\infty}\leq C\delta,\ \ \ \ \forall t>0.
\end{align*}
	This result is based on the construction of the fundamental solution to the heat equation with $BV$ coefficient,
	\begin{align*}
		&\partial_t \mathbf{K}(t,x,y)-\partial_{x}(\mu(x)\partial_{x}\mathbf{K}(t,x,y))=0,\\
		&\lim_{t\to 0^+}\mathbf{K}(t,x,y)=\mathbf{Dirac}(x-y),
	\end{align*}
for any $x, y\in \mathbb{R}$, $x\not= y$ where $\mu(x)$ is a $BV$ function with the property $
	\inf _x \mu(x)\geq c>0$ and $||\mu-1||_{BV}\ll 1.$\vspace{0.1cm}\\
Later on, the result  was extended to the full compressible Navier--Stokes system \eqref{cpns} by  Wang--Yu--Zhang \cite{Wang,Wang1}. They proved the global well-posedness for initial data $(v_0,u_0,\theta_0)$ satisfying 
\begin{align}\label{theircon2}
	\|(v_0-1,u_0,\theta_0-1)\|_{L^1\cap BV}\leq \delta,
\end{align}
		for sufficiently small $\delta$, the solution
exists global-in-time and satisfies 
\begin{align*}
	\|(v(t)-1,u(t),\theta(t)-1)\|_{L^1\cap BV}&+\sqrt{t+1}\|(v(t)-1,u(t),\theta(t)-1)\|_{L^\infty}\\
&\quad\quad\quad\quad\quad\quad\quad\quad\quad+\sqrt{t}\|(u_x(t),\theta_x(t))\|_{L^\infty}\leq C\delta,\ \ \ \  \forall t>0.
\end{align*}
	In the  more recently work \cite{CHN}, based on some smoothing and Lipschitz-type estimates,  the authors proved local  well-posedness of both system \eqref{inns} and \eqref{cpns} with very rough data, which allows $v_0$ to have finite jumps.  A simple example of $v_0$ is given for readers' convenience: Consider $v_0$ that is ``almost continuous" in $(-\infty,0)\cup(0,+\infty)$, with the jump $|v_0(0^+)-v_0(0^-)|$ that can be arbitrarily large. More precisely, ``almost continuous" means that there exists $\varepsilon>0$ such that $|v_0(x)-v_0(y)|\leq b$ for any $x,y$ satisfying $xy>0, |x-y|\leq \varepsilon$, where  $b$ is a small constant.  Let  us mention the local results briefly. For simplicity,  we denote $\|\cdot\|_{L^p_{T,x}}$ the $L^p$ norm in $[0,T]\times\mathbb{R}$, and $\|\cdot\|_{L^p}$ the $L^p$ norm in $\mathbb{R}$.
	For  initial data $(v_0,u_0)$ satisfying 
	\begin{align*}
	   v_0\in L^\infty,\quad \quad \inf_{x}v_0\geq \lambda_0>0,\quad\quad \|u_0\|_{\dot B^{-1+2\gamma}_{\infty,\infty}}< \infty,
	\end{align*}
the system \eqref{inns} admits a unique local solution  $(v,u)$ in $[0,T]$ satisfying
		\begin{align*}			&\inf_{t\in[0,T]}\inf_xv(t,x)\geq \frac{\lambda_0}{2},\ \ \ \ \sup_{t\in[0,T]}(\|v(t)\|_{L^\infty}+ t^{1-\gamma}\|u_x(t)\|_{L^\infty})
		<\infty.	
		\end{align*}
	Here $0< \gamma \ll 1$ is fixed constant that can be taken arbitrarily small.\\
	For the full compressible Navier--Stokes system \eqref{cpns}, we consider initial data $(v_0,u_0,\theta_0)$ that  satisfies
	\begin{align}\label{idcns}
		\inf_xv_0(x)\geq \lambda_0>0,\ \ \	\|v_0\|_{L^\infty}<\infty,\ \ \|u_0\|_{L^2}<\infty,\ \ \|\theta_0\|_{\dot W^{-\frac{2}{3},\frac{6}{5}}}<\infty.
	\end{align}
	There exists a unique local solution $(v,u,\theta)$ in $[0,T]$ satisfying 
	\begin{align*}
	 \inf_{t\in[0,T]}\inf_xv(t,x)\geq \frac{\lambda_0}{2},\ \ \  \|v\|_{L^\infty_{T,x}}+ \|u_x\|_{L^2_{T,x}}+\|\theta\|_{L^2_{T,x}}+\sup_{t\in[0,T]} t^{\frac{3}{4}}\|u_x(t) \|_{L^\infty}<\infty.
	\end{align*}
Here we define, for any $\beta\in (0,1),p\in [1,\infty)$, 
\begin{align}\label{defneg}
    ||g||_{\dot W^{-\beta,p}}:=\inf\{ ||f||_{\dot W^{1-\beta,p}}: g=\partial_x f\},
\end{align}
where 
\begin{align*}
    ||f||_{\dot W^{1-\beta,p}}:=\left(\iint_{\mathbb{R}^2}\frac{|f(x)-f(y)|^p}{|x-y|^{1+(1-\beta)p}}dxdy\right)^\frac{1}{p}.
\end{align*}
Note that $\dot W^{-\beta,p}$ does not coincide with the dual space of $\dot W^{\beta,\frac{p}{p-1}}$ for any $p\not =2$.

	In this paper, we improve the global well-posedness result of Liu and Yu with initial velocity data in  $W^{2\gamma,1}$ space; and of Wang-Yu-Zhang with initial velocity data in $ L^2\cap W^{2\gamma,1}$ space and initial data of temperature in $\dot W^{-\frac{2}{3},\frac{6}{5}}\cap \dot W^{2\gamma-1,1}$ for any $\gamma>0$ that can be \textit{arbitrarily small}. We state our main results as follows.
	\begin{thm}\label{thm1} Let $\gamma\in (0,\frac{1}{100}].$
	There exists $\varepsilon_0>0$ such that any initial data $(v_0,u_0)$ satisfying 
	\begin{align}\label{glocon1}
	 M_1:=  \|v_0-1\|_{L^1\cap BV}+\|u_0\|_{W^{2\gamma,1}}\leq\varepsilon_0, 
	\end{align}
	the system \eqref{inns} admits a unique global  solution $(v,u)$  such that  
\begin{align}\label{esthm1}
&\|(v(t)-1,u(t))\|_{L^1\cap BV}+\sqrt{t+1}\|(v(t)-1,u(t))\|_{L^\infty}+\sqrt{t}\|u_x(t)\|_{L^\infty}\lesssim  (1+\frac{1}{t})M_1,\ \ \ \ \forall t>0.
\end{align}

	\end{thm}
		\begin{thm}\label{thm2} Let $\gamma\in (0,\frac{1}{100}].$
			There exists $\varepsilon_1>0$ such that any initial data $(v_0,u_0,\theta_0)$ satisfying 
	\begin{align}\label{glocon2}
	   M_2:=\|v_0-1\|_{L^1\cap BV}+\|u_0\|_{W^{2\gamma,1}\cap L^2}+\|\theta_0-1\|_{\dot W^{-\frac{2}{3},\frac{6}{5}}\cap \dot W^{2\gamma-1,1}}\leq\varepsilon_1,
	\end{align}
	the system \eqref{inns} admits a unique global  solution $(v,u,\theta)$ such that  for $U(t)=(v(t)-1,u(t),\theta(t)-1)$
	\begin{equation}\label{main2}
	    \begin{aligned}
	        &	\|U(t)\|_{L^1\cap BV}+\sqrt{t+1}\|U(t)\|_{L^\infty}+\sqrt{t}\|(u_x(t),\theta_x(t))\|_{L^\infty}\lesssim (1+\frac{1}{t})M_2,\ \ \ \  \forall t>0.
	    \end{aligned}
	\end{equation}
\end{thm}
The idea to prove Theorem \ref{thm1} and \ref{thm2} is from \cite{Peskin}. The key point is to  shows that the solution becomes more regular in a short time $T$. Then, start from time $T$, we can apply results in Liu--Yu \cite{LiuCAPM} and  Wang--Yu--Zhang \cite{Wang,Wang1} to obtain global existence.\vspace{0.1cm}\\
	We make some remarks about the setting of initial data. We note that $\|v_0-1\|_{L^1}+\|v_0\|_{BV}\leq\varepsilon$ implies that $\|v_0-1\|_{L^\infty}\leq \varepsilon$, this provides a positive lower bound of $v_0$. Compare with previous global results (see \cite{LiuCAPM,Wang1}), we relax the regularity of initial data $u_0$ in \eqref{inns} and $(u_0,\theta_0)$ in \eqref{cpns}, respectively. More precisely, for system \eqref{inns}, we relax the condition $u_0\in BV$ in \eqref{theircon} to $u_0\in W^{2\gamma,1}$, which improved almost 1 derivative in view of the fact that $BV$ has the same scaling with $\dot W^{1,1}$. We note that $\dot W^{2\gamma,1}$ has the same scaling with $\dot B^{-1+2\gamma}_{\infty,\infty}$, this shows the consistency of our global result and local result. Furthermore, we take the inhomogeneous Sobolev norm $W^{2\gamma,1}$ that also contains $L^1$ norm, which helps us to control lower frequencies. Moreover, for system \eqref{cpns}, we relax the condition $u_0, \theta_0-1\in BV\cap L^1$ in \eqref{theircon2} to $u_0\in L^2\cap W^{2\gamma,1}, \theta_0-1\in \dot W^{-\frac{2}{3},\frac{6}{5}}\cap \dot W^{2\gamma-1,1}$. Recall that the initial condition for local well-posedness is $u_0\in L^2, \theta_0\in \dot W^{-\frac{2}{3},\frac{6}{5}}$. By Sobolev embedding, we have $\dot W^{-\frac{2}{3},\frac{6}{5}}\subset\dot W^{2\gamma-1,\frac{2}{1+2\gamma}}$, which indicates that the new constraint $\theta_0-1\in \dot W^{2\gamma-1,1}$ is a lower order control. 
	
	The improvements are done because of the smoothing effect of parabolic equations. However, the hyperbolic nature of $v$ will not generate any smoothing effect. Hence, by known results, it is natural to set $v_0$ with $\|v_0-1\|_{L^1}+\|v_0\|_{BV}\ll 1$. 
	To propogate $\|v(t)\|_{BV}$, it requires $\|u_x\|_{L^1_t BV_x}$ bounded. So, by the parabolic nature of $u$, $u_0$ should satisfy $$\int_{0}^{1}||\partial_x\left( e^{t\partial_{x}^2}u_0\right)||_{BV}dt<\infty.$$
	This motivates us to take $u_0\in W^{2\gamma,1}$.
	 For the full N--S system \eqref{cpns}, we keep the condition  $u_0\in L^2$ and $\theta_0-1\in \dot W^{-\frac{2}{3},\frac{6}{5}}$ from the  local result (see \eqref{idcns}). Due to the similar structure with \eqref{inns}, we also need $\|u_x\|_{L^1_t BV_x}$. A key observation is that, $\theta$ is expected to have the same regularity with $u_x$. This can be seen from the equation of $u$ in \eqref{cpns}, where we have the term $\partial_x\left(\frac{K\theta}{v}-\frac{\mu u_x}{v}\right)$. Roughly speaking, $\theta$ and $u_x$ play the same role in system  \eqref{cpns}. Hence we need $\|\theta\|_{L^1_tBV_x}$. Moreover, due to the parabolic nature of the equation of $\theta$ in \eqref{cpns}, as above, 
	we impose the condition   $\|\theta_0-1\|_{\dot W^{2\gamma-1,1}}$, which helps us to obtain $\|\theta\|_{L^1_tBV_x}$.
	\vspace{0.2cm}\\
The main difficulty of our paper is to study regularity theory for the parabolic equation  with $BV$ coefficient: 
\begin{equation}\label{toymodel}
		\begin{aligned}
			&\partial_t f(t,x) -\partial_x(A(t,x)\partial_xf(t,x))=0,\\
			&f(0,x)=f_0(x).
		\end{aligned}	
	\end{equation}
	where $A(t,x)=\frac{1}{v(t,x)}$ satisfies  $||A-1||_{L^\infty_T BV_x}<<1$ for some $T>0$ small.\\
	We want to have the regularity $$f_x\in L^1_T BV_x.$$ By the standard argument of regularity theory, one can prove that if $||A||_{L^\infty_t W^{1+\kappa_1,1}}<\infty,$ for some $\kappa_1>0$ one has 
	\begin{align}\label{LBV}
	  \int_0^T ||f_x(t)||_{ BV_x}dt\lesssim ||f_0||_{W^{\kappa,1}},
	\end{align}
	for any $0<\kappa<\kappa_1$ and $T>0$ small.  \\
	However, we would like to prove \eqref{LBV} with $\kappa_1=0$. It is an end-point estimate. Fortunately, we can have H\"older regularity of $v(t,x)$ in time since $v(t,x)=u_0+\int_0^t u_x(\tau,x)d\tau$. This allows us to prove \eqref{LBV} with $||A||_{C^{\alpha}_t BV_x}<\infty$ for some $\alpha\in (0,\frac{1}{16})$. Indeed, one can write  
	\begin{equation}\label{toymodel'}
		\begin{aligned}
			&\partial_t f(t,x) -\partial_x^2f(t,x))=\partial_x F,~~F=(A-1)\partial_xf,\\
			&f(0,x)=f_0(x).
		\end{aligned}	
	\end{equation}
	Now we introduce a new norm: 
	\begin{equation}\label{zzz}
	   \|h\|_{BV_{T}^{\sigma}}=\sup_{0<s<T}s^{\sigma}\|f(s)\|_{BV}+\sup_{0<s<t<T}s^{\sigma+\alpha}\frac{\|f(t)-f(s)\|_{BV}}{(t-s)^\alpha}.
	\end{equation}
	We can prove that for any $\kappa\in (0,1)$
	\begin{equation*}
	   ||f_x||_{BV_{T}^{1-\frac{\kappa}{2}}}\lesssim  ||f_0||_{W^{\kappa,1}}+ ||F||_{BV_{T}^{1-\frac{\kappa}{2}}},
	\end{equation*}
	see \eqref{BV} in Lemma \ref{lemma}. Note that we can not drop the second norm in \eqref{zzz} since  the following estimate 
	\begin{equation*}
	   \sup_{0<s<T} s^{1-\frac{\kappa}{2}}||f_x(s)||_{BV}\lesssim  ||f_0||_{W^{\kappa,1}}+  \sup_{0<s<T} s^{1-\frac{\kappa}{2}}||F(s)||_{BV}
	\end{equation*}
	 is not available. \vspace{0.1cm}\\
	By  $||A||_{C^{\alpha}_t BV_x}<\infty$  and letting $T$ small enough, $||F||_{BV_{T}^{1-\frac{\kappa}{2}}}$ can be absorbed by $||f_x||_{BV_{T}^{1-\frac{\kappa}{2}}}$ in the left hand-side, so one gets
	\begin{equation*}
	   ||f_x||_{BV_{T}^{1-\frac{\kappa}{2}}}\lesssim  ||f_0||_{W^{\kappa,1}},
	\end{equation*}
 This implies \eqref{LBV}. \vspace{0.2cm}\\
The rest of this paper is organized as follows. We establish the main Lemma \ref{lemma} in Section \ref{secmain}. In Section 3 we apply Lemma \ref{lemma} to obtain global well-posedness for system (\ref{inns}) and prove Theorem \ref{thm1}. Section 4 is devoted to prove global well-posedness for  system (\ref{cpns}) and prove Theorem \ref{thm2}. 
	
	\section{Main lemma and its proof}\label{secmain}
	We first introduce some notations in the proof. 	The $BV$ norm is defined as 
$$		\|h\|_{BV}:=\sup_{P\in\mathcal{P}}\sum _{i\in\mathbb{Z}}|h(x_{i+1})-h(x_i)|,$$where the supremum is taken over the set of partition  $\mathcal{P}=\{P=\{x_i\}_{i\in\mathbb{Z}}: x_i\leq x_{i+1}\}$ of $\mathbb{R}$. 	By the definition of $BV$ norm, it is easy to check that for any function $h_1\in L^1(\mathbb{R}), h_2\in BV(\mathbb{R})$, $f\in\dot W^{1,1}$,
\begin{equation}\label{bvcon}
	\|h_1\ast h_2\|_{BV}\lesssim \|h_1\|_{L^1}\|h_2\|_{BV},
\end{equation}
and
\begin{equation}\label{embed}
		\|f\|_{BV}\leq\|f_x\|_{L^1}.
	\end{equation}
For $b\in(0,1)$,	we denote the fractional Laplacian operator $\Lambda^b=(-\Delta)^\frac{b}{2}$, which is the Fourier multiplier with symbol $|\xi|^{b}$. And denote $\Lambda^{-b}$ the Fourier multiplier with symbol $|\xi|^{-b}$. By Plancherel theorem, we have 
\begin{align}\label{move}
\int_{\mathbb{R}}f(x)g(x)dx=\int_{\mathbb{R}}\Lambda^bf(x)\Lambda^{-b}g(x)dx.
\end{align}
	
		Throughout the paper, we fix two constants $0<\alpha<\gamma\ll1$. For $\sigma\geq 0$, $f:(0,T)\times \mathbb{R}\to \mathbb{R}$, define the norms
	\begin{align*}
		&	\|f\|_{X_T^{\sigma,p}}=\sup_{0<s<T}s^{\sigma}\|f(s)\|_{L^p}+\sup_{0<s<t<T}s^{\sigma+\alpha}\frac{\|f(t)-f(s)\|_{L^p}}{(t-s)^\alpha},\\
		&\|f\|_{BV_{T}^{\sigma}}=\sup_{0<s<T}s^{\sigma}\|f(s)\|_{BV}+\sup_{0<s<t<T}s^{\sigma+\alpha}\frac{\|f(t)-f(s)\|_{BV}}{(t-s)^\alpha}.
	\end{align*}
 Without confusion of notations, we denote $\|\cdot\|_{L^p_{T,x}}$ the $L^p$ norm in $[0,T]\times\mathbb{R}$, and $\|\cdot\|_{L^p}$ the $L^p$ norm in $\mathbb{R}$. Moreover, we write $A\lesssim B$ if there exists a universal constant $C$ such that $A\leq CB$. 
	We will solve the system (\ref{inns}) and \eqref{cpns} by the fixed point theorem in the Banach space equipped with the norms above. We note that the norms $\|\cdot\|_{X_T^{\sigma,p}}$ and $\|\cdot\|_{BV^\sigma_T}$ contain  H\"{o}lder derivatives in time.\\
	In this section, we consider the Cauchy problem of the parabolic equation:
	\begin{equation}\label{eqpara}
		\begin{aligned}
			&\partial_t f(t,x) -\partial_x^2f(t,x)=\partial_x F(t,x)+R(t,x),\\
			&f(0,x)=f_0(x).
		\end{aligned}	
	\end{equation}
The solution has formula 
\begin{align}
	f(t,x)&=\int_{\mathbb{R}}\mathbf{K}(t,x-y) f_0(y)dy+\int_0^t \int_{\mathbb{R}}\mathbf{K}(t-\tau,x-y)\partial_x F(\tau,y)dyd\tau\nonumber\\
	&\quad\quad\quad \ \ \ + \int_0^t \int_{\mathbb{R}}\mathbf{K}(t-\tau,x-y)R(\tau,y)dyd\tau\nonumber\\&:=f_L(t,x)+f_N(t,x)+f_R(t,x),\label{forso}
\end{align}
	where 
 $$\mathbf{K}(t,x)=(4\pi t)^{-\frac{1}{2}}e^{-\frac{x^2}{4t}}$$  is the standard heat kernel in $\mathbb{R}$.
	The following estimates for  heat kernel will be used frequently in our proof.
	\begin{lemma}\label{lemheat}
		There holds 
		\begin{align*}
			&|\partial_t^j\partial_x^m \mathbf{K}(t,x)|\lesssim \frac{1}{(t^\frac{1}{2}+|x|)^{1+2j+m}},\\ &\|\partial_t^j\partial_x^m\mathbf{K}(t,\cdot)\|_{L^p}\lesssim t^{-j+\frac{1}{2}(\frac{1}{p}-1-m)},\ \ \ \ \|\partial_x^l\Lambda^{-\sigma}\mathbf{K}(t,\cdot)\|_{L^p}\lesssim t^{\frac{1}{2}(\frac{1}{p}-1-l+\sigma)},\\
			&\| \partial_x^m\mathbf{K}(t+a,\cdot)-\partial_x^m \mathbf{K}(t,\cdot)\|_{L^p}\lesssim \frac{1}{t^{{\frac{1}{2}(m-\frac{1}{p}+1)}}}\min\left\{1,\frac{a}{t}\right\},
		\end{align*}
		for any $m,j=0,1,2$, $l=1,2$, any $\sigma\in(0,1),a,t\geq 0$ and $p\in[1,+\infty]$.
	\end{lemma}
	It is easy to check the above estimates by the definition of heat kernel and the fact that $b^me^{-b}\lesssim_m 1$, $\forall b>0,m\in\mathbb{N}^+$. We omit details here.

	We recall the following Schauder type lemmas from \cite{CHN}, the idea can be found in  \cite{Peskin}.
	\begin{lemma}\label{mainlem}
		Let $\mathbf{K}$ be the heat kernel, and let
		\begin{align*}
			g(t,x)=\int_0^t\int_{\mathbb{R}}\partial_t \mathbf{K}(t-\tau,x-y) f(\tau,y)dyd\tau.
		\end{align*}
		Then we have
		\begin{align*}
			\|g\|_{X_T^{\sigma,p}}\lesssim \|f\|_{X_T^{\sigma,p}},\ \ \ \forall T>0,\  \sigma\in(0,1-\alpha),\  p\in[1,\infty].
		\end{align*}
		And for $\sigma\in(1,\frac{3}{2})$, we have
		\begin{equation*}
		\|g\|_{X_T^{\sigma,\infty}}\lesssim \|f\|_{X_T^{\sigma,\infty}}+\|f\|_{X_T^{\sigma-\frac{1}{2},1}}.
		\end{equation*}
	\end{lemma}
\begin{proof}
The first estimate follows directly from \label{mainlem}\cite[Lemma 2.3]{CHN}.
		We only prove the latter estimate. One has	
		\begin{align*}
		   	g(t)&=\int_0^\frac{t}{2}\partial_t \mathbf{K}(t-\tau)\ast f(\tau)d\tau +\int_{t/2}^t\partial_t \mathbf{K}(t-\tau)\ast (f(\tau)-f(t))d\tau+\int_{t/2}^t\partial_t \mathbf{K}(t-\tau)\ast f(t)d\tau\\
		&:=g_1(t)+g_2(t)+g_3(t).
		\end{align*}
		For $g_1$, by Lemma \ref{lemheat} we have 
		\begin{align*}
		\|g_{1}(t)\|_{L^\infty}\lesssim \int_0^\frac{t}{2}\|\partial_t\mathbf{K}(t-\tau)\|_{L^\infty}\|f(\tau)\|_{L^1}d\tau\lesssim t^{-\frac{3}{2}}\int_0^\frac{t}{2}\tau^{\frac{1}{2}-\sigma}d\tau\|f\|_{X_T^{\sigma-\frac{1}{2},1}}\lesssim t^{-\sigma}\|f\|_{X_T^{\sigma-\frac{1}{2},1}}
		\end{align*}
		Similarly, for $g_{2}$, we have 
		\begin{align*}
		\|g_{2}(t)\|_{L^\infty}
		&\lesssim \int_{\frac{t}{2}}^{t}\|\partial_t\mathbf{K}(t-s)\|_{L^1}\|f(t)-f(s)\|_{L^\infty}ds\\
		&\lesssim \int_{\frac{t}{2}}^{t}(t-s)^{-1+\alpha}s^{-\sigma-\alpha}ds\|f\|_{X_T^{\sigma,\infty}}\lesssim t^{-\sigma}\|f\|_{X_T^{\sigma,\infty}}.
		\end{align*}
		Moreover, observe that $g_3(t)=\mathbf{K}(t/2)\ast f(t)-f(t)$, then 
		$$
		\|g_3(t)\|_{L^\infty}\lesssim (1+\|\mathbf{K}(t/2)\|_{L^1})\|f\|_{L^\infty}\lesssim t^{-\sigma}\|f\|_{X^{\sigma,\infty}_T}.
		$$
		Thus, we get
		\begin{equation}\label{1}
		\sup_{0<t<T}t^\sigma\|g(t)\|_{L^\infty}\lesssim \|f\|_{X_T^{\sigma,\infty}}+\|f\|_{X_T^{\sigma-\frac{1}{2},1}}
.		\end{equation}	
		In the following we will denote $a=t-s>0$ for convenience, and we denote $\delta_\beta f(s,x)=f(s+\beta,x)-f(s,x)$. We write
		\begin{align*}
		g(t)-	g(s)=&\int_0^s\delta_a\partial_t \mathbf{K}(s-\tau)\ast f(\tau)d\tau+\int_s^t\partial_t \mathbf{K}(t-\tau)\ast f(\tau)d\tau\\
		=&\int_0^s\delta_a\partial_t \mathbf{K}(s-\tau)\ast (f(\tau)-f(s))d\tau+\int_s^t\partial_t \mathbf{K}(t-\tau)\ast (f(\tau)-f(t))d\tau\\
		&+\big(\int_0^s\delta_a\partial_t \mathbf{K}(s-\tau)\ast f(s)d\tau+\int_s^t\partial_t \mathbf{K}(t-\tau)\ast f(t)d\tau\big)\\
		:=&I_1+I_2+I_3.
		\end{align*}
		For $I_1$, we still use the decomposition that
		\begin{align*}
		I_1= \left(\int_0^{\frac{s}{2}}+\int_{\frac{s}{2}}^{s}\right)\delta_a\partial_t\mathbf{K}(t-\tau)\ast (f(\tau)-f(s))d\tau=I_{11}+I_{12}.
.		\end{align*}
		For $I_{11}$, we apply Lemma \ref{lemheat} to obtain
		\begin{align*}
		\|I_{11}\|_{L^\infty}&\lesssim\int_0^{\frac{s}{2}}\|\delta_a\partial_t\mathbf{K}(s-\tau)\|_{L^\infty}\|f(\tau)-f(s)\|_{L^1}d\tau\lesssim a^\alpha s^{-\frac{3}{2}-\alpha}\int_0^\frac{s}{2}\tau^{\frac{1}{2}-\sigma}d\tau\|f\|_{X_T^{\sigma-\frac{1}{2},1}}\\
		&\lesssim a^\alpha s^{-\sigma-\alpha}\|f\|_{X_T^{\sigma-\frac{1}{2},1}}.		\end{align*}
		For $I_{12}$, we have 
		\begin{align*}
		\|I_{12}\|_{L^\infty}&\lesssim \int_{\frac{s}{2}}^{s}\|\delta_a\partial_t\mathbf{K}(s-\tau)\|_{L^1}\|f(s)-f(\tau)\|_{L^\infty}ds\\
		&\lesssim s^{-\sigma-\alpha}\int_{\frac{s}{2}}^{s}(s-\tau)^{-1+\alpha}\min\{1,\frac{a}{s-\tau}\}d\tau\|f\|_{X_T^{\sigma,\infty}}\\
		&\lesssim a^\alpha s^{-\sigma-\alpha}\|f\|_{X_T^{\sigma,\infty}}.
		\end{align*}
		For $I_2$, we have 
		\begin{align*}
		\|I_2\|_{L^\infty}&\lesssim \int_s^t \|\partial_t \mathbf{K}(t-\tau)\|_{L^1}\|f(\tau)-f(t)\|_{L^\infty}d\tau\\
		&\lesssim \int_s^t \frac{1}{(t-\tau)^{1-\alpha}}\frac{1}{\tau^{\sigma+\alpha}}d\tau \|f\|_{X_T^{\sigma,\infty}}\lesssim a^\alpha s^{-\sigma-\alpha}\|f\|_{X_T^{\sigma,\infty}}.
		\end{align*}
		For $I_3$, we have
		\begin{align*}
		\|I_3\|_{L^\infty}\lesssim &\|f(s)\ast \mathbf{K}(t-s)-f(s)\ast \mathbf{K}(t)-f(s)+f(s)\ast \mathbf{K}(s)+f(t)-f(t)\ast \mathbf{K}(t-s)\|_{L^\infty}\\
		\lesssim &\|f(t)-f(s)\|_{L^\infty}+\|f(s)\ast(\mathbf{K}(t)- \mathbf{K}(s))\|_{L^\infty}
		\lesssim s^{-\sigma-\alpha}a^\alpha\|f\|_{X_T^{\sigma,\infty}},
		\end{align*}
		where we applied Lemma \ref{lemheat} to get $\|\mathbf{K}(t)- \mathbf{K}(s)\|_{L^1}\lesssim \frac{(t-s)^\alpha}{s^\alpha}$ in the last inequality. Hence we obtain that
	\begin{equation*}
		\sup_{0<s<t<T}\frac{s^{\sigma+\alpha}}{a^\alpha}\|g(t)-g(s)\|_{L^\infty}\lesssim \|f\|_{X_T^{\sigma,\infty}}+\|f\|_{X_T^{\sigma-\frac{1}{2},1}}.
		\end{equation*}	
	Combining this with \eqref{1}, we complete the proof.
	\end{proof}
		\vspace{0.2cm}\\
The following lemma is a simplified version of \cite[Lemma 2.4]{CHN}.
	\begin{lemma}\label{lemlow}
		Let $$g(t,x)=\int_0^t\int_{\mathbb{R}} \partial_x \mathbf{K}(t-\tau,x,y) f(\tau,y)dyd\tau.$$ Then for $\sigma,\tilde \sigma\in(0,1)$, $\sigma\geq \tilde \sigma-\frac{1}{2}$, we have 
		$$
		\|g\|_{X_T^{\sigma,p}}\lesssim T^{\frac{1}{2}+\sigma-\tilde \sigma }\|f\|_{X_T^{\tilde{\sigma},p}},\ \ \ \forall \ 1\leq p\leq \infty.
		$$
	\end{lemma}
	
	\begin{lemma}\label{rem}
		Suppose 
		$$
		g(t,x)=\int_0^t\int_{\mathbb{R}}\mathbf{K}(t-\tau,x-y)R(\tau,y)dyd\tau.
		$$
		Let $l=0,1$. Then for any  $\frac{1}{1+2\alpha}<p<\frac{1}{2\alpha}$ and $\sigma\geq \frac{1+l}{2}-\frac{1}{2p}$,  there holds 
		$$
		\|\partial_x^lg\|_{X_T^{\sigma,p}}\lesssim T^{\sigma-\frac{1+l}{2}+\frac{1}{2p}}(\|R\|_{L_{T}^1}+\|R\|_{X_T^{1,1}}).
		$$
	\end{lemma}
	And for $\sigma\in(1,\frac{3}{2})$, we have 
	\begin{equation*}
	\|\partial_xg\|_{X_T^{\sigma,\infty}}\lesssim T^{\sigma-1}(\|R\|_{L_{T,x}^1}+\|R\|_{X_T^{\frac{3}{2},\infty}}).
	\end{equation*}
	\begin{proof} The first estimate follows directly from \cite[Lemma 2.6]{CHN}.
		We only prove the latter estimate. Obviously we have
		\begin{equation*}
		\|\partial_xg(t)\|_{L^\infty}\lesssim \|(\int_0^{\frac{t}{2}}+\int_{\frac{t}{2}}^{t})\partial_x\mathbf{K}(t-s)\ast R(s)ds\|_{L^\infty}\lesssim \|g_1\|_{L^\infty}+\|g_2\|_{L^\infty}.
		\end{equation*}
		For $g_1$ we have 
		\begin{equation*}
		\|g_1(t)\|_{L^\infty}\lesssim \int_0^\frac{t}{2}\|\partial_x\mathbf{K}(t-s)\|_{L^\infty}\|R(s)\|_{L^1}ds\lesssim t^{-1}\|R\|_{L_{T,x}^1}.
		\end{equation*}
		For $g_2$ we have
		\begin{equation*}
		\|g_2(t)\|_{L^\infty}\lesssim \int_{\frac{t}{2}}^{t}\|\partial_x\mathbf{K}(t-s)\|_{L^1}\|R(s)\|_{L^\infty}ds\lesssim t^{-1}\|R\|_{X_T^{\frac{3}{2},\infty}}.
		\end{equation*}
		So we have 
		\begin{equation*}
		\sup_{0<t<T}t^{\sigma}\|\partial_xg(t)\|_{L^\infty}\lesssim T^{\sigma-1}(\|R\|_{L_{T,x}^1}+\|R\|_{X_T^{\frac{3}{2},\infty}}).
		\end{equation*}
		Then we take time difference. We denote $a=t-s$. If $a\geq \frac{s}{4}$, it is easy to check that 
		\begin{align*}
		\sup_{0<s<t<T}s^{\sigma+\alpha}\frac{\|\partial_xg(t)-\partial_xg(s)\|_{L^\infty}}{a^\alpha}&\lesssim \sup_{0<s<t<T}s^{\sigma}(\|\partial_xg(t)\|_{L^\infty}+\|\partial_xg(s)\|_{L^\infty})\\
		&\lesssim  T^{\sigma-1}(\|R\|_{L_{T,x}^1}+\|R\|_{X_T^{\frac{3}{2},\infty}}).
		\end{align*}
		Then it suffices to prove the case $a<\frac{s}{4}$. We have
		\begin{align*}
		\partial_xg(t)-\partial_xg(s)=\int_0^s\delta_a\partial_x\mathbf{K}(s-\tau)\ast R(\tau)d\tau+\int_{s}^{t}\partial_x\mathbf{K}(t-\tau)\ast R(\tau)d\tau:=J_1+J_2.
		\end{align*}
		For $J_1$, we have 
		\begin{equation*}
		\|J_1\|_{L^\infty}\lesssim \|(\int_0^\frac{s}{2}+\int_{\frac{s}{2}}^{s})\partial_x\delta_a\mathbf{K}(s-\tau)\ast R(\tau)d\tau\|_{L^\infty}\lesssim\|J_{11}\|_{L^\infty}+\|J_{12}\|_{L^\infty}.
		\end{equation*}
		For $J_{11}$, we have 
		\begin{align*}
		\|J_{11}\|_{L^\infty}&\lesssim \int_{0}^{\frac{s}{2}}\|\partial_x\delta_a\mathbf{K}(s-\tau)\|_{L^\infty}\|R(\tau)\|_{L^1}ds\lesssim \int_{0}^{\frac{s}{2}}(s-\tau)^{-1-\alpha}a^\alpha\|R(\tau)\|_{L^1}ds\\
		&\lesssim a^\alpha s^{-1-\alpha}\|R\|_{L_{T,x}^1}.
		\end{align*}
		For $J_{12}$, we have 
		\begin{align*}
		\|J_{12}\|_{L^\infty}&\lesssim \int_{\frac{s}{2}}^{s}\|\partial_x\delta_a\mathbf{K}(s-\tau)\|_{L^1}\|R(\tau)\|_{L^\infty}d\tau\lesssim\int_{\frac{s}{2}}^{s}(s-\tau)^{-\frac{1}{2}-\alpha}a^\alpha \tau^{-\frac{3}{2}}\|R\|_{X_T^{\frac{3}{2},\infty}}d\tau\\
		&\lesssim s^{-1-\alpha}a^\alpha\|R\|_{X_T^{\frac{3}{2},\infty}}.
		\end{align*}
		For $J_2$, we have 
		\begin{equation*}
		\|J_2\|_{L^\infty}\lesssim \int_s^t\|\partial_x\mathbf{K}(t-\tau)\|_{L^1}\|R(\tau)\|_{L^\infty}\lesssim a^\alpha s^{-1-\alpha}\|R\|_{X_T^{\frac{3}{2},\infty}}.
		\end{equation*}
		Combine the estimates all above gives
		\begin{equation*}
		\sup_{0<s<t<T}s^{\sigma+\alpha}\frac{\|\partial_xg(t)-\partial_xg(s)\|_{L^\infty}}{(t-s)^\alpha}\lesssim T^{\sigma-1}(\|R\|_{L_{T,x}^1}+\|R\|_{X_T^{\frac{3}{2},\infty}}),
		\end{equation*}
		and the proof is done.
	\end{proof}\\
	
The following lemma is  a $BV$ version of Lemma \ref{mainlem}.
\begin{lemma}\label{mainlemBV}
		Let $\mathbf{K}$ be the heat kernel, and let
		\begin{align*}
			g(t,x)=\int_0^t\int_{\mathbb{R}}\partial_t \mathbf{K}(t-\tau,x-y) f(\tau,y)dyd\tau.
		\end{align*}
		Then we have
		\begin{align}\label{re}
			\|g\|_{BV^{\sigma}_T}\lesssim \|f\|_{BV^{\sigma}_T},\ \ \ \forall T>0,\  \sigma\in(0,1-\alpha).
		\end{align}
	\end{lemma}
\begin{proof}
	The idea of  the proof follows \cite[Lemma 2.3]{CHN}. 
We can write that
\begin{equation*}
	g(t,x)=\int_0^t\partial_t\mathbf{K}(t-s)\ast(f(s)-f(t))ds+\int_0^t\partial_t\mathbf{K}(t-s)\ast f(t)ds=g_1(t,x)+g_2(t,x).
\end{equation*}
For $g_1$, applying Lemma \ref{lemheat} to obtain that for any $0<t<T$,
\begin{equation}\label{fnsg1}
	\begin{aligned}
		\|g_1(t)\|_{BV}\lesssim &\int_0^t \|\partial_t \mathbf{K}(t-s)\|_{L^1}\|f(s)-f(t)\|_{BV}ds\\
		\lesssim& \|f\|_{BV_T^\sigma}\int_0^t(t-s)^{-1+\alpha}s^{-\sigma-\alpha}ds\\
		\lesssim& t^{-\sigma}\|f\|_{BV_T^\sigma}.
	\end{aligned}
\end{equation}
For $g_2$, using the fact that $\lim_{t\to 0}\mathbf{K}(t,x)=\mathbf{Dirac}(x)$, we have 
\begin{equation*}
	g_2(t,x)=\int_{\mathbb{R}} \mathbf{K}(t,x-y)f(t,y)dy-f(t,x).
\end{equation*}
By \eqref{bvcon} we obtain 
\begin{equation}\label{fnsg2}
	\|g_2(t)\|_{BV}\lesssim  \|f(t)\|_{BV}.
\end{equation}
 We get from \eqref{fnsg1} and \eqref{fnsg2} that 
 \begin{align}\label{bv1}
 	\sup_{t\in[0,T]}t^\sigma\|g(t)\|_{BV}\lesssim \|f\|_{BV^\sigma_T}.
 \end{align}
Now we consider the time difference. Take  $0<s<t<T$. Denote $a=t-s$ and $\delta_a h(s,x)=h(s+a,x)-h(s,x)$, then 
\begin{align*}
	g(t)-	g(s)=&\int_0^s\delta_a\partial_t \mathbf{K}(s-\tau)\ast f(\tau)d\tau+\int_s^t\partial_t \mathbf{K}(t-\tau)\ast f(\tau)d\tau\\
	=&\int_0^s\delta_a\partial_t \mathbf{K}(s-\tau)\ast (f(\tau)-f(s))d\tau+\int_s^t\partial_t \mathbf{K}(t-\tau)\ast (f(\tau)-f(t))d\tau\\
	&+\left(\int_0^s\delta_a\partial_t \mathbf{K}(s-\tau)\ast f(s)d\tau+\int_s^t\partial_t \mathbf{K}(t-\tau)\ast f(t)d\tau\right)\\
	:=&I_1+I_2+I_3.
\end{align*}
By Lemma \ref{lemheat}, we can see that
\begin{align*}
&	\|I_1\|_{BV}\lesssim a^\alpha s^{-\sigma-\alpha}\|f\|_{BV_T^\sigma},\\
&	\|I_2\|_{BV}\lesssim \int_s^t(t-\tau)^{-1+\alpha}s^{-\sigma-\alpha}d\tau\|f\|_{BV_T^\sigma}\lesssim a^\alpha s^{-\sigma-\alpha}\|f\|_{BV_T^\sigma}.
\end{align*}
For $I_3$, we have that
\begin{equation*}
	\begin{aligned}
		I_3&=\mathbf{K}(t-s)\ast f(s)-f(s)-\mathbf{K}(t)\ast f(s)+\mathbf{K}(s)\ast f(s)+f(t)-\mathbf{K}(t-s)\ast f(t)\\
		&=(f(t)-f(s))-\mathbf{K}(t-s)\ast (f(t)-f(s))-(\mathbf{K}(t)-\mathbf{K}(s))\ast f(s).
	\end{aligned}
\end{equation*}
Hence,
\begin{align*}
	\|I_3\|_{BV}\lesssim \|f(t)-f(s)\|_{BV}+\|\mathbf{K}(t)-\mathbf{K}(s)\|_{l^1}\|f(s)\|_{BV}\lesssim s^{-\sigma-\alpha} a^\alpha \|f\|_{BV_T^\sigma}.
\end{align*}
	where we applied Lemma \ref{lemheat} to get $\|\mathbf{K}(t)- \mathbf{K}(s)\|_{L^1}\lesssim \frac{(t-s)^\alpha}{s^\alpha}$ in the last inequality.
Combining the estimates of  $I_i, i=1,2,3$, we have
\begin{equation}\label{fnstd}
	\sup_{0<s<t<T}s^{\sigma+\alpha}\frac{\|g(s)-g(t)\|_{BV}}{(t-s)^{\alpha}}\lesssim \|f\|_{BV_T^\sigma}.
\end{equation}
Then we finish the proof in view of \eqref{bv1} and \eqref{fnstd}.
\end{proof}\vspace{0.3cm}\\
We present the main lemma as follows.
	\begin{lemma}\label{lemma}Let $f $ be a solution to \eqref{eqpara}. Then for any $T>0$, there holds 
		\begin{align}
			&\|f\|_{L^2_{T,x}\cap X_T^{\frac{1}{2},2}}+	\|\partial_x f\|_{L^\frac{6}{5}_{T,x}\cap X_T^{\frac{5}{6},\frac{6}{5}}}\lesssim \|f_0\|_{\dot W^{-\frac{2}{3},\frac{6}{5}}}+ \| F\|_{L^\frac{6}{5}_{T,x}\cap X_T^{\frac{5}{6},\frac{6}{5}}}+T^{\frac{1}{4}}\|R\|_{L^1_{T,x}\cap X_T^{1,1}},\label{the1}\\
		&\|f\|_{X^{\frac{1}{2}-\gamma,1}_T}+\|\partial_xf\|_{X^{1-\gamma,1}_T}\lesssim \|f_0\|_{\dot W^{2\gamma-1,1}}+\|F\|_{X^{1-\gamma,1}_T}+T^{\frac{1}{2}-\gamma}\|R\|_{L^1_{T,x}\cap X_T^{1,1}}.\label{l1}
		\end{align}
		Moreover, if $R=0$, we have 
		\begin{align}
		&\|\partial_xf\|_{L^2_{T,x}\cap X_T^{\frac{1}{2},2}\cap X_T^{\frac{3}{4},\infty}}\lesssim \|f_0\|_{L^2}+\|F\|_{L^2_{T,x}\cap X_T^{\frac{1}{2},2}\cap X_T^{\frac{3}{4},\infty}},\label{u1}\\
		&\|\partial_x f\|_{X_T^{1-\gamma,\infty}}\lesssim \|f_0\|_{\dot W^{2\gamma,1}}+\|F\|_{X_T^{1-\gamma,\infty}},\label{uinf}\\
		&\|\partial_x f\|_{X_T^{\frac{1}{2}-\gamma,1}}\lesssim \|f_0\|_{\dot W^{2\gamma,1}}+\|F\|_{X_T^{\frac{1}{2}-\gamma,1}},\label{l2}\\
			&\|\partial_x f\|_{BV_{T}^{1-\gamma}}\lesssim \|f_0\|_{\dot W^{2\gamma,1}}+\| F\|_{BV_{T}^{1-\gamma}}\label{BV}.
		\end{align}
	\end{lemma}
	\begin{remark}
	We note that the estimates \eqref{uinf}, \eqref{l2} and \eqref{BV} are used to obtain the local existence of system \eqref{inns}. In system \eqref{cpns}, the estimates \eqref{the1} and \eqref{l1} will be used to estimate $\theta$, the estimates \eqref{u1}, \eqref{l2} and \eqref{BV} will be used to estimate $u$.
	\end{remark}
	\begin{proof}
	It suffices to prove \eqref{l1},\eqref{l2} and \eqref{BV}, and  the proof of other estimates can be found in \cite[Lemma 3.1]{CHN}. Recall the formula of solution \eqref{forso}. We first estimate the linear part $f_L$, we claim that  
	\begin{align}
		&	\|f_L\|_{X^{\frac{1}{2}-\gamma,1}_T}+\|\partial_xf_L\|_{X^{1-\gamma,1}_T}\lesssim \|f_0\|_{\dot W^{2\gamma-1,1}},\label{li2}	\\	&	\|\partial_x f_L\|_{X_T^{1-\gamma,\infty}}+\|\partial_x^2 f_L\|_{X_{T}^{1-\gamma,1}}+\|\partial_x f_L\|_{X_{T}^{\frac{1}{2}-\gamma,1}}\lesssim\|f_0\|_{\dot W^{2\gamma,1}}\label{li1}.
	\end{align}
	We first prove \eqref{li2}. By definition \eqref{defneg}, it suffices to prove 
	$$
		\|f_L\|_{X^{\frac{1}{2}-\gamma,1}_T}+\|\partial_xf_L\|_{X^{1-\gamma,1}_T}\lesssim \|\bar f_0\|_{\dot W^{2\gamma,1}},
	$$
	for $\bar f_0$ that satisfies $\partial_x \bar f_0=f_0$. We integrate by parts and apply \eqref{move} to obtain 
	\begin{align*}
	   f_L(t,x)=\int_{\mathbb{R}}\partial_x\mathbf{K}(t,x-y)\bar f_0(y)dy=\int_{\mathbb{R}}\Lambda^{-2\gamma}\partial_x\mathbf{K}(t,x-y)\Lambda^{2\gamma}\bar f_0(y)dy.
	\end{align*}
Let $k=0,1$. 	By Young's inequality and Lemma \ref{lemheat}, we have for $0<s<t<T$,
\begin{align*}
  &\|\partial_x^kf_L(t)\|_{L^1}\lesssim \|\Lambda^{-2\gamma}\partial_x^{1+k}\mathbf{K}(t)\|_{L^1}\|\Lambda^{2\gamma}\bar f_0\|_{L^1}\lesssim t^{-\frac{1+k}{2}+\gamma}\|\bar f_0\|_{\dot W^{2\gamma,1}},\\
	   &\|\partial_x^kf_L(t)-\partial_x^kf_L(s)\|_{L^1}\lesssim \|\Lambda^{-2\gamma}\partial_x^{1+k}\mathbf{K}(t)-\Lambda^{-2\gamma}\partial_x^{1+k}\mathbf{K}(s)\|_{L^1}\|\Lambda^{2\gamma}\bar f_0\|_{L^1}\\
	   &\quad\quad\quad\ ~ \quad\quad\quad\quad\quad\quad\lesssim s^{-\frac{1+k}{2}+\gamma}\min\{1,\frac{t-s}{s}\}\|\bar f_0\|_{W^{2\gamma,1}}\lesssim (t-s)^\alpha s^{-\frac{1+k}{2}+\gamma-\alpha}\|\bar f_0\|_{W^{2\gamma,1}}.
	\end{align*}
	This yields \eqref{li2}. 
Then we prove \eqref{li1}. By the same arguments to prove \eqref{li2} with $\bar f_0$ replaced by $f_0$, we obtain that 
\begin{align}\label{1111}
\|\partial_x^2 f_L\|_{X_{T}^{1-\gamma,1}}+\|\partial_x f_L\|_{X_{T}^{\frac{1}{2}-\gamma,1}}\lesssim\|f_0\|_{\dot W^{2\gamma,1}}.
\end{align}
Moreover,  we have 
	\begin{align*}
		&	\|	\partial_x	f_L(t)\|_{L^\infty}\lesssim  \|\partial_x\Lambda^{-2\gamma}\mathbf{K}(t)\|_{L^\infty}\|\Lambda^{2\gamma}f_0\|_{L^1}\lesssim t^{\gamma-1}\|f_0\|_{\dot W^{2\gamma,1}},\\
&			\|\partial_x	f_L(t)-\partial_x	f_L(s)\|_{L^\infty}\lesssim \|\partial_x\Lambda^{-2\gamma}\mathbf{K}(t)-\partial_x\Lambda^{-2\gamma}\mathbf{K}(s)\|_{L^\infty}\|\Lambda^{2\gamma}g_0\|_{L^1}\\
			&\ \quad\quad\quad\ ~ \quad\quad\quad\quad\quad\quad\lesssim s^{\gamma-1}\min\{1,\frac{t-s}{s}\}\|\Lambda^{2\gamma}f_0\|_{L^1}\lesssim(t-s)^\alpha s^{\gamma-1-\alpha}\|f_0\|_{\dot W^{2\gamma,1}}.
	\end{align*}
Then we obtain 
$$
\|\partial_x f_L\|_{X^{1-\gamma,\infty}}\lesssim \|f_0\|_{\dot W^{2\gamma,1}}.
$$
This and \eqref{1111} yield \eqref{li1}. \\
Applying Lemma \ref{mainlem} with $(\sigma,p)\in\{(1-\gamma,1),(\frac{1}{2}-\gamma,1)\}$, and Lemma \ref{lemlow} with $(\sigma,\tilde \sigma,p)=(\frac{1}{2}-\gamma,1-\gamma,1)$ to obtain that 
\begin{align}\label{fN}
&\|f_N\|_{X^{\frac{1}{2}-\gamma,1}_T}+\|\partial_xf_N\|_{X^{1-\gamma,1}_T}\lesssim \|F\|_{X_T^{1-\gamma,1}},\\
&\|\partial_xf_N\|_{X^{1-\gamma,\infty}_T}\lesssim \|F\|_{X_T^{1-\gamma,\infty}},\ \ \ \ \|\partial_xf_N\|_{X^{\frac{1}{2}-\gamma,1}_T}\lesssim \|F\|_{X_T^{\frac{1}{2}-\gamma,1}}.\label{fn3}
\end{align}
Similarly, by Lemma \ref{mainlemBV}, one gets 
\begin{align}\label{fn2}
\|\partial_x f_N\|_{BV^{1-\gamma}_T}\lesssim \|F\|_{BV^{1-\gamma}_T}.
\end{align}
Finally, Lemma \ref{rem} implies that 
\begin{align}\label{fR}
\|f_R\|_{X^{\frac{1}{2}-\gamma,1}_T}+\|\partial_xf_R\|_{X^{1-\gamma,1}_T}\lesssim T^{\frac{1}{2}-\gamma}(\|R\|_{L^1_T}+\|R\|_{X_T^{1,1}}).
\end{align}
Then we obtain \eqref{l1} from \eqref{li2}, \eqref{fN}, \eqref{fR}. And \eqref{li1}, \eqref{fn3} yield \eqref{l2}.  Finally, \eqref{BV} is a result of \eqref{li1}, \eqref{fn2}, together with the fact that $\|\partial_x f_L\|_{BV_T^{1-\gamma}}\lesssim \|\partial_x^2 f_L\|_{X_T^{1-\gamma,1}}$. This completes the proof.
\end{proof}\vspace{0.3cm}\\
We introduce the following lemmas, which are helpful to estimate the nonlinear terms.
\begin{lemma}\label{lemproduct}
For any $\sigma\in[0,1], p\in[1,\infty]$, and any function $g,h:[0,T]\times \mathbb{R}\to \mathbb{R}$, there holds 
\begin{align*}
\|gh\|_{X_T^{\sigma,p}}\lesssim \|g\|_{X_T^{0,\infty}}\|h\|_{X_T^{\sigma,p}},\quad\quad \ \ \|gh\|_{BV^\sigma_T}\lesssim (\|g\|_{BV^0_T}+\|g\|_{X_T^{0,\infty}})(\|h\|_{BV^\sigma_T}+\|h\|_{X^{\sigma,\infty}_T}).
\end{align*}
\end{lemma}
\begin{proof}
It is easy to check that 
		\begin{align*}
		\sup_{t\in[0,T]}t^{\sigma}\|gh(t)\|_{L^p}&\lesssim \|g\|_{L_{T,x}^{\infty}}\sup_{t\in[0,T]}t^{\sigma}\|h(t)\|_{L^{p}}\lesssim\|g\|_{X_T^{0,\infty}}\|h\|_{X_T^{\sigma,p}},\\
		\sup_{t\in[0,T]}t^{\sigma}\|gh(t)\|_{BV}&\lesssim \|g\|_{L_{T,x}^{\infty}}\sup_{t\in[0,T]}t^{\sigma}\|h(t)\|_{BV}+\sup_t\|g(t)\|_{BV}\sup_t t^\sigma \|f(t)\|_{L^\infty}\\
		&\lesssim\|g\|_{X_T^{0,\infty}}\|h\|_{BV_T^{\sigma}}+\|g\|_{BV_T^{0}}\|h\|_{X_T^{\sigma,\infty}}.
		\end{align*}
		Moreover, we have 
		\begin{equation*}
		\begin{aligned}
	&	\sup_{0<s<t<T}s^{\sigma+\alpha}\frac{\|gh(t)-gh(s)\|_{L^p}}{(t-s)^\alpha}\\
		&\lesssim \|g\|_{L_{T,x}^\infty}\sup_{0<s<t<T}s^{\sigma+\alpha}\frac{\|h(t)-h(s)\|_{L^{p}}}{(t-s)^{\alpha}}+\sup_{s\in[0,T]}\left(s^{\sigma}\|h(s)\|_{L^{p}}\right)\sup_{0<s<t<T}\left(\frac{s^\alpha\|g(t)-g(s)\|_{L^\infty}}{(t-s)^\alpha}\right)\\
		&\lesssim \|g\|_{X_T^{0,\infty}}\|h\|_{X_T^{\sigma,p}},
		\end{aligned}
		\end{equation*}
and 
		\begin{equation*}
		\begin{aligned}
		&\sup_{0<s<t<T}s^{\sigma+\alpha}\frac{\|gh(t)-gh(s)\|_{BV}}{(t-s)^\alpha}\\
		\lesssim& \sup_{0<s<t<T}s^{\sigma+\alpha}\|g(t)\|_{BV}\frac{\|h(t)-h(s)\|_{L^\infty}}{(t-s)^\alpha}+\sup_{0<s<t<T}s^{\sigma+\alpha}\|g(t)\|_{L^\infty}\frac{\|h(t)-h(s)\|_{BV}}{(t-s)^\alpha}\\
		&+\sup_{s<t<T}s^{\sigma+\alpha}\|h(s)\|_{BV}\frac{\|g(t)-g(s)\|_{L^\infty}}{(t-s)^\alpha}+\sup_{s<t<T}s^{\sigma+\alpha}\|h(s)\|_{L^\infty}\frac{\|g(t)-g(s)\|_{BV}}{(t-s)^\alpha}\\
		\lesssim& (\|g\|_{BV^0_T}+\|g\|_{X_T^{0,\infty}})(\|h\|_{BV^\sigma_T}+\|h\|_{X^{\sigma,\infty}_T}).
		\end{aligned}
		\end{equation*}
This completes the proof.
\end{proof}\vspace{0.3cm}\\
For $I_1,I_2\subset\mathbb{R}$, denote  the convex hull $\text{Conv}\{I_1,I_2\}=\{x\in\mathbb{R}:\exists x_1\in I_1,x_2\in I_2, r\in[0,1]\ \text{such that }\ x=rx_1+(1-r)x_2\}$. We have the following lemma.
\begin{lemma}\label{lemfra}
    Consider functions $g,h:[0,T]\times \mathbb{R}\to \mathbb{R}$, and $B:\mathbb{R}\to \mathbb{R}$ that  satisfy $$\max_{k=0,1,2,3}\frac{d^kB(z)}{dz^k}\leq \lambda_0, \ \ \ \forall z\in \text{Conv}\{ \operatorname{range}(g),\operatorname{range}(h)\},$$ then there holds 
    \begin{align*}
  &  \left\|B(g)-B(h)\right\|_{X^{0,\infty}_T}\lesssim \|g-h\|_{X^{0,\infty}_T}(1+\|g\|_{X^{0,\infty}_T}+\|h\|_{X^{0,\infty}_T}),\\
  &  \left\|B(g)-B(h)\right\|_{X^{0,1}_T}\lesssim \|g-h\|_{X^{0,1}_T}(1+\|g\|_{X^{0,\infty}_T}+\|h\|_{X^{0,\infty}_T}),\\
   & \left\|B(g)-B(h)\right\|_{BV^{0}_T}\lesssim (\|g-h\|_{BV^{0}_T}+\|g-h\|_{X^{0,\infty}_T})(1+\|g\|_{BV^{0}_T}+\|h\|_{BV^{0}_T}+\|g\|_{X^{0,\infty}_T}+\|h\|_{X^{0,\infty}_T})^2,
    \end{align*}
    where the implicit constants depend only on $\lambda_0$.
\end{lemma}
\begin{proof}
The first two estimates follows direct from the inequalities 
\begin{align*}
&\left\|B(g(t))-B(h(t))\right\|_{L^\infty}\lesssim \|g(t)-h(t)\|_{L^\infty},\\
&\left\|B(g(t))-B(h(t))-B(g(s))+B(h(s))\right\|_{L^\infty}\\
&\quad\quad\quad\lesssim \|(g-h)(t)-(g-h)(s)\|_{L^\infty}+\|(g-h)(s)\|_{L^\infty}(\|g(t)-g(s)\|_{L^\infty}+\|h(t)-h(s)\|_{L^\infty}),
\end{align*}
and 
\begin{align*}
&\left\|B(g(t))-B(h(t))\right\|_{L^1}\lesssim \|g(t)-h(t)\|_{L^1},\\
&\left\|B(g(t))-B(h(t))-B(g(s))+B(h(s))\right\|_{L^1}\\
&\quad\quad\quad\lesssim \|(g-h)(t)-(g-h)(s)\|_{L^1}+\|(g-h)(s)\|_{L^1}(\|g(t)-g(s)\|_{L^\infty}+\|h(t)-h(s)\|_{L^\infty}).
\end{align*}
For the last estimate, by the definiton  of $BV$ norm, it is easy to check that 
\begin{align*}
&\left\|B(g(t))-B(h(t))\right\|_{BV}\lesssim\|g(t)-h(t)\|_{BV}+ \|g(t)-h(t)\|_{L^\infty}(\|g(t)\|_{BV}+\|h(t)\|_{BV}),
\end{align*}
and 
\begin{align*}
&\left\|B(g(t))-B(h(t))-B(g(s))+B(h(s))\right\|_{BV}\\
&\lesssim \|(g-h)(t)-(g-h)(s)\|_{BV}+\|(g-h)(s)\|_{BV}(\|g(t)-g(s)\|_{L^\infty}+\|h(t)-h(s)\|_{L^\infty})\\
&+\|(g-h)(s)\|_{L^\infty}(\|g(t)-g(s)\|_{BV}+\|h(t)-h(s)\|_{BV})\\
&+(\|(g-h)(t)-(g-h)(s)\|_{L^\infty}+\|(g-h)(s)\|_{L^\infty}(\|g(t)-g(s)\|_{L^\infty}+\|h(t)-h(s)\|_{L^\infty}))\\
&\quad\times(\|g(t)\|_{BV}+\|h(t)\|_{BV}+\|g(s)\|_{BV}+\|h(s)\|_{BV}).
\end{align*}
Finally, by the definition of the norms in $X^{0,1}_T,X^{0,\infty}_T, BV^0_T$, we obtain the result.
\end{proof}
	\section{Well-posedness for the isentropic Navier–Stokes equations}\label{jum}
This section is devoted to prove Theorem \ref{thm1}. 
We construct a solution to \eqref{inns} by Banach fixed point theorem.
Define  the norms
\begin{align*}
&	\|v\|_{Y_T}:=	\|v\|_{X^{0,1}_T}+	\|v\|_{X_T^{0,\infty}}+\|v\|_{BV_T^0},\\
&	\|w\|_T:=\|\partial_x w\|_{X_T^{1-\gamma,\infty}}+\|\partial_x w\|_{X_T^{\frac{1}{2}-\gamma,1}}+\|\partial_x w\|_{BV^{1-\gamma}_T}.
\end{align*}
The main result is the following 
\begin{thm}\label{thmisen}
	There exist $\varepsilon_0,T\in(0,1)$ such that, for any  initial data $(v_0,u_0)$ satisfying
	\begin{align}\label{glocon1}
	   	 M_1:=  \|v_0-1\|_{L^1\cap BV}+\|u_0\|_{W^{2\gamma,1}}\leq\varepsilon_0, 
	\end{align} 
	the system \eqref{inns} admits a unique  solution $(v,u)$ satisfying \begin{align}\label{rethm3}
	   \|v-1\|_{Y_T}+\|u\|_{T}\lesssim M_1.
	\end{align}
\end{thm}
\begin{remark}
We remark that Theorem \ref{thmisen} does not require the explicit formula $p(v)=Av^\nu, \ 1\leq \nu<e$. Indeed, we only need $p\in W^{3,\infty}([\frac{1}{2},\frac{3}{2}])$.
\end{remark}
\begin{proof}
For $T,\eta>0$, define the space 
\begin{align*}
	\mathcal{E}_{T,\eta}:=\{w:w(0,x)=u_0(x),\|w\|_T\leq \eta\}.
\end{align*}
To prove the existence of solution to \eqref{inns}, we construct the solution as a fixed point of the map $\mathcal{T}$, which is defined as follows.\\
Consider $w\in\mathcal{E}_{T,\eta}$, we define a map $\mathcal{T}w=u$, where $u$ is a solution to the equation 
\begin{align}\label{eqjum}
	\partial_tu-\mu\partial_x^2 u&=-(p(v)-p(1))_x+\mu\left(\left(\frac{1}{v}-1\right)\partial_x w\right)_x=:\tilde F(v,w)_x,
\end{align}
where $v$ is defined by 
\begin{align}\label{defv}
	v(t,x)=v_0(x)+\int_0^t \partial_x w(s,x)ds.
\end{align}
We claim that for any $c>0$, there exists $\varepsilon_0\in(0,1) $ such that if $M_1\leq\varepsilon_0$, then
\begin{align}
		& \|\mathcal{T}w\|_{T}\leq C_0(1+ cT^{{\gamma}})M_1,\quad\quad\quad\forall w \in \mathcal{E}_{T,cM_1},\label{es1}\\
		&\|\mathcal{T}w_1-\mathcal{T}w_2\|_{T}\leq  C_0(  T^\gamma +M_1)\|w_1-w_2\|_{T}, \quad\quad \ \forall w_1,w_2\in \mathcal{E}_{T,cM_1},\label{es2}
	\end{align}
	hold for any $T\in (0,1)$, and a constant $C_0$ independent of $c$.\\\\
We first prove \eqref{es1}. 
Let $w\in \mathcal{E}_{cM_1}$.	We first show the estimates of $v$.
By definition \eqref{defv}, it is easy to check that 
	\begin{align*}
		&\sup_{s\in[0,T]}\|v(s)-1\|_{L^1}\leq \|v_0-1\|_{L^1}+\int_0^T\|w_{x}(\tau)\|_{L^1}d\tau\leq \|v_0-1\|_{L^1}+2T^{\frac{1}{2}+\gamma}\|\partial_x w\|_{X_T^{\frac{1}{2}-\gamma,1}},\\
 	&\sup_{0<s<t<T}s^\alpha\frac{\|v(t)-v(s)\|_{L^1}}{(t-s)^\alpha}\lesssim \sup_{0<s<t<T}s^\alpha\frac{\int_s^t\|\partial_x w\|_{L^1}d\tau}{(t-s)^\alpha}\\
		&\quad\ \quad\quad\quad\quad\quad\quad\quad \ \ \ \ \ \ \ \ \ \ \ \lesssim \|\partial_x w\|_{X_T^{\frac{1}{2}-\gamma,1}}\sup_{s<t<T}s^\alpha \frac{t^{\gamma+\frac{1}{2}}-s^{\gamma+\frac{1}{2}}}{(t-s)^\alpha}\lesssim T^{\frac{1}{2}+\gamma}\|\partial_x w\|_{X_T^{\frac{1}{2}-\gamma,1}}.
	\end{align*}
	Similarly,
	\begin{align*}
		&\sup_{s\in[0,T]}\| v(s)\|_{BV}\leq 	\| v_0\|_{BV}+\int_0^T\|\partial_x w(\tau)\|_{BV} d\tau\leq	\|v_0\|_{BV}+\gamma^{-1} T^\gamma \|\partial_x w\|_{BV_T^{1-\gamma}},\\
		&\sup_{0<s<t<T}s^\alpha\frac{\|v(t)-v(s)\|_{BV}}{(t-s)^\alpha}\lesssim \sup_{0<s<t<T}s^\alpha\frac{\int_s^t\|\partial_x w\|_{BV}d\tau}{(t-s)^\alpha}\\
		&\quad\ \quad\quad\quad\quad\quad\quad\quad \ \ \ \ \ \ \ \ \ \ \ \lesssim \|\partial_x w\|_{BV_T^{1-\gamma}}\sup_{s<t<T}s^\alpha \frac{t^\gamma-s^\gamma}{(t-s)^\alpha}\lesssim T^\gamma\|\partial_x w\|_{BV_T^{1-\gamma}}.
	\end{align*}
	Hence we get 
	\begin{align*}
		&\|v-1\|_{X^{0,1}_T}\lesssim\|v_0-1\|_{L^1}+T^{\frac{1}{2}+\gamma}\|\partial_x w\|_{X_T^{\frac{1}{2}-\gamma,1}}\lesssim (1+cT^{\frac{1}{2}+\gamma})M_1,\\
		&\| v\|_{BV_T^{0}}\lesssim\| v_0\|_{BV}+T^\gamma \|\partial_x w\|_{BV_T^{1-\gamma}} \lesssim(1+cT^{\gamma})M_1.
	\end{align*}
	Moreover, using the fact that $\|f\|_{L^\infty}\leq \|f\|_{L^1}+\|f\|_{BV}$, we can obtain $$
	\|v-1\|_{X_T^{0,\infty}}\leq \|v-1\|_{X^{0,1}_T}+\| v\|_{BV_T^{0}}.
	$$
Then 
	\begin{align}\label{estofv}
	\|v-1\|_{Y_T}=	\|v-1\|_{X^{0,1}_T}+	\|v-1\|_{X_T^{0,\infty}}+\|v\|_{BV_T^0}\leq C_1 (1+cT^{\gamma})M_1.
	\end{align}	
	By taking $\varepsilon_0<\frac{1}{(10+C_1+c)^5}$, we have $\|v-1\|_{Y_T}\leq \frac{1}{2}$ provided $M_1\leq \varepsilon_0$.
	In particular, from the above estimate we have 
	$$\inf_{t\in[0,T]}\inf _{x\in\mathbb{R}}v(t,x)\geq \frac{1}{2},\quad\quad\left\|\frac{1}{v}\right\|_{L^\infty_{T,x}}\leq 2.$$ 
	Thanks to  Lemma \ref{lemma}, we have
	\begin{align}
		&\| \partial_x(\mathcal{T}w)\|_{X^{\frac{1}{2}-\gamma,1}_T}+\|\partial_x (\mathcal{T}w)\|_{BV_T^{1-\gamma}}+	\|\partial_x (\mathcal{T}w)\|_{X_T^{1-\gamma,\infty}}\nonumber\\
		&\quad\quad\lesssim \| u_0\|_{\dot W^{2\gamma,1}}+\|\tilde F(v,w)\|_{X^{\frac{1}{2}-\gamma,1}_T}+\|\tilde F(v,w)\|_{BV_T^{1-\gamma}}+\|\tilde F(v,w)\|_{X_T^{1-\gamma,\infty}}, \ \ \ \forall T>0.\label{haha}
	\end{align}
Note that  
\begin{align*}
	&\|\tilde F(v,w)\|_{\star}\lesssim \|p(v)-p(1)\|_{\star}+\left\|\left(\frac{1}{v}-1\right)\partial_x w\right\|_{\star},\ \ \quad\quad \star\in\{X^{\frac{1}{2}-\gamma,1}_T,X^{1-\gamma,\infty}_T,BV^{1-\gamma}_T\}.
\end{align*}
Applying Lemma \ref{lemfra} to obtain that 
\begin{align}
&\|p(v)-p(1)\|_{X^{\frac{1}{2}-\gamma,1}_T}+\|p(v)-p(1)\|_{X^{1-\gamma,\infty}_T}+\|p(v)-p(1)\|_{BV^{1-\gamma}_T}\nonumber\\
&\lesssim (T^{\frac{1}{2}-\gamma}+T^{1-\gamma})(\|p(v)-p(1)\|_{X^{0,1}_T}+\|p(v)-p(1)\|_{X^{0,\infty}_T}+\|p(v)-p(1)\|_{BV^{0}_T})\nonumber\\
&\lesssim T^{\frac{1}{2}-\gamma}\|v-1\|_{Y_T}(1+\|v-1\|_{Y_T})^2\overset{\eqref{estofv}}\lesssim T^{\frac{1}{2}-\gamma}(1+cT^{\gamma})M_1.\label{p1}
\end{align}
	Denote $q(t,x)=(\frac{1}{v(t,x)}-1)\partial_x w(t,x)$, by Lemma \ref{lemproduct} and \eqref{estofv} we have 
	\begin{align*}
&\|q\|_{X_T^{\frac{1}{2}-\gamma,1}}+\|q\|_{X_T^{1-\gamma,\infty}}+\|q\|_{BV_T^{1-\gamma,\infty}}\\
&\lesssim \left(\left\|\frac{1}{v}-1\right\|_{X^{0,\infty}_T}+\left\|\frac{1}{v}-1\right\|_{BV^{0}_T}\right)(\|\partial_xw\|_{X_T^{\frac{1}{2}-\gamma,1}}+\|\partial_xw\|_{X_T^{1-\gamma,\infty}}+\|\partial_x w\|_{BV^{1-\gamma}_T})\\
&\lesssim \|v-1\|_{Y_T}(1+\|v-1\|_{Y_T})\|\partial_xw\|_{T}\lesssim c(1+cT^{\gamma})M_1^2.
	\end{align*}
	Combining this with \eqref{p1} to obtain that 
	\begin{align}
	&	\|\tilde F(v,w)\|_{X_T^{\frac{1}{2}-\gamma,1}}+	\|\tilde F(v,w)\|_{X_T^{1-\gamma,\infty}}+	\|\tilde F(v,w)\|_{BV_T^{1-\gamma}}\nonumber\\
		&\lesssim  (T^{\frac{1}{2}-\gamma}+cM_1)(1+cT^{\gamma})M_1\lesssim (1+cT^{\gamma})M_1.\label{tilF}
	\end{align}
	Here we use the fact that $T^{\frac{1}{2}-\gamma}+cM_1\leq T^{\frac{1}{2}-\gamma}+c\eps_0\lesssim 1$.
	We conclude from   \eqref{haha} and \eqref{tilF}  that 
	\begin{align*}
		\|\mathcal{T}w\|_{T}&\lesssim(1+cT^{\gamma})M_1,\ \ \ \ \ \forall w\in \mathcal{E}_{T,cM_1}.
	\end{align*}
This completes the proof of  \eqref{es1}. Then we prove \eqref{es2}. 
	 Consider $w_1,w_2\in 	\mathcal{E}_{T,cM_1}$, let $$v_m(t,x)=v_0(x)+\int_0^t \partial_x w_m(s,x)ds,\ \ \ m=1,2.$$ Then it is easy to check that both $v_1$ and $v_2$ satisfy \eqref{estofv}, and 
	\begin{equation}\label{difv}
	   	\begin{aligned}
	&\|v_1-v_2\|_{ Y_T}=\|v_1-v_2\|_{X_T^{0,\infty}}+\|v_1-v_2\|_{X_T^{0,1}}+\|v_1-v_2\|_{BV_T^{0}}\\
	&\lesssim T^{\frac{1}{2}+\gamma}\|\partial_x (w_1-w_2)\|_{X_T^{\frac{1}{2}-\gamma,1}}+T^\gamma \|\partial_x (w_1-w_2)\|_{X_T^{1-\gamma,\infty}}+ T^\gamma\|\partial_x (w_1-w_2)\|_{BV_T^{1-\gamma}}\\
	&\lesssim T^\gamma \|w_1-w_2\|_T.
	\end{aligned}
	\end{equation}
	We have the equation 
	\begin{align*}
		\partial_t(\mathcal{T}w_1-	\mathcal{T}w_2)-\mu\partial_x^2 (\mathcal{T}w_1-	\mathcal{T}w_2)&=(\tilde F(v_1,w_1)-\tilde F(v_2,w_2))_x.
	\end{align*}
By Lemma \ref{lemma}, we obtain 
	\begin{align}\label{contra}
		\|\mathcal{T}w_1-	\mathcal{T}w_2\|_{T}\lesssim& \|\tilde F(v_1,w_1)-\tilde F(v_2,w_2)\|_{X_T^{1-\gamma,\infty}}+ \|\tilde F(v_1,w_1)-\tilde F(v_2,w_2)\|_{BV_T^{1-\gamma}}\nonumber\\
		&+\|\tilde F(v_1,w_1)-\tilde F(v_2,w_2)\|_{X_T^{\frac{1}{2}-\gamma,1}}.
	\end{align}
	Note that 
	\begin{align*}
		\tilde	F(v_1,w_1)-\tilde F(v_2,w_2)=p(v_2)-p(v_1)+\mu \left(\left(\frac{1}{v_1}-1\right)\partial_xw_1- \left(\frac{1}{v_2}-1\right)\partial_xw_2\right).
	\end{align*}
	We first estimate $p(v_2)-p(v_1)$. Applying Lemma \ref{lemproduct} and \eqref{difv} to obtain  that 
	\begin{align}
	   &\|p(v_2)-p(v_1)\|_{X_T^{\frac{1}{2}-\gamma,1}}+\|p(v_2)-p(v_1)\|_{X_T^{1-\gamma,\infty}}+\|p(v_2)-p(v_1)\|_{BV_T^{1-\gamma}}\nonumber\\
	   &\quad\quad\lesssim T^{\frac{1}{2}-\gamma} (\|p(v_2)-p(v_1)\|_{X^{0,1}_T}+\|p(v_2)-p(v_1)\|_{X^{0,\infty}_T}+\|p(v_2)-p(v_1)\|_{BV^{0}_T})\nonumber\\
	   &\quad\quad\lesssim T^{\frac{1}{2}-\gamma}\|v_1-v_2\|_{ Y_T}(1+\|v_1-1\|_{Y_T}+\|v_2-1\|_{Y_T})^2\lesssim T^\frac{1}{2}\|w_1-w_2\|_T.\label{difp}
	\end{align}
	Then we estimate $\mathcal{Q}=\left(\frac{1}{v_1}-1\right)\partial_xw_1- \left(\frac{1}{v_2}-1\right)\partial_xw_2$. We have 
	\begin{align*}
		\mathcal{Q}=\left(\frac{1}{v_1}-\frac{1}{v_2}\right)\partial_xw_1+\left(\frac{1}{v_2}-1\right)\partial_x(w_1-w_2).
	\end{align*}
By Lemma \ref{lemproduct} and Lemma \ref{lemfra},
\begin{align}
&\|\mathcal{Q}\|_{X^{1-\gamma,\infty}_T}+\|\mathcal{Q}\|_{X^{\frac{1}{2}-\gamma,1}_T}+\|\mathcal{Q}\|_{BV^{1-\gamma}_T}\nonumber\\
&\lesssim \left(\left\|\frac{1}{v_1}-\frac{1}{v_2}\right\|_{X^{0,\infty}_T}+\left\|\frac{1}{v_1}-\frac{1}{v_2}\right\|_{BV^{0}_T}\right)\|w_1\|_T+\left(\left\|\frac{1}{v_2}-1\right\|_{X^{0,\infty}_T}+\left\|\frac{1}{v_2}-1\right\|_{BV^{0}_T}\right)\|w_1-w_2\|_T\nonumber\\
&\lesssim \|v_1-v_2\|_{ Y_T}(1+\|v_1-1\|_{Y_T}+\|v_2-1\|_{ Y_T})^2\|w_1\|_T+\|v_2-1\|_{Y_T}(1+\|v_2-1\|_{Y_T})^2\|w_1-w_2\|_T\nonumber\\
&\lesssim \|w_1-w_2\|_T(1+cT^\gamma)M_1,\label{Q}
\end{align}
where the last inequality follows from \eqref{difv}, and the facts  that  $\|v_2-1\|_{Y_T}\lesssim (1+cM_1)T^\gamma, \|w_1\|_T\leq cM_1$.
	We conclude from \eqref{contra}, \eqref{difp} and \eqref{Q} that 
	\begin{align*}
		\|\mathcal{T}w_1-	\mathcal{T}w_2\|_{T}\lesssim(  T^\frac{1}{2}+(1+cT^\gamma)M_1)\|w_1-w_2\|_{T}\lesssim (T^\gamma+M_1)\|w_1-w_2\|_{T}.
	\end{align*}
	This completes the proof of\eqref{es2}.\\\\
	Without loss of generality, we further assume $0<\varepsilon_0\leq \frac{1}{C_0(10+c)^2}$.	With the claim \eqref{es1} and \eqref{es2}, we can take $c_1=10C_0$ and  $0<T<(10+C_0)^{-\frac{10}{\gamma}}$ to obtain  that $$
	\|\mathcal{T}w\|_{T}\leq c_1M_1,\ \ \ \|\mathcal{T}w_1-\mathcal{T}w_2\|_{T}\leq \frac{1}{2}\|w_1-w_2\|_{T},\ \ \ \text{for}\ w,w_1,w_2\in \mathcal{E}_{T,c_1M_1}.
	$$
	Then by classical Bananch fixed point theorem, the map $\mathcal{T}$ has a unique fixed point $u\in\mathcal{E}_{T,c_1M_1}$. Let $v(t)=v_0+\int_0^t u_x(\tau)d\tau$. Clearly, $(v,u)$ solves system \eqref{inns} in $[0,T]$ with initial data $(v_0,u_0)$. The estimate \eqref{rethm3} follows directly from \eqref{estofv} and the fact that $u\in\mathcal{E}_{T,c_1M_1}$. This completes the proof.
\end{proof}\vspace{0.3cm}\\
\begin{proof}[Proof of Theorem \ref{thm1}]
From \eqref{rethm3} in  Theorem \ref{thmisen} we obtain that for initial data satisfying \eqref{glocon1}, there exists a unique solution $(v,u)$ in $[0,T]$ such that 
\begin{align}
\sup_{t\in[0,T]}(\|v(t)-1\|_{L^1\cap L^\infty\cap BV}+t^{\frac{1}{2}-\gamma}\|u(t)\|_{ BV}+t^{1-\gamma}\|u_x(t)\|_{L^\infty})\lesssim M_1.\label{re11}
\end{align}
Moreover, using the formula 
\begin{align*}
u(t)=\mathbf{K}(t)\ast u_0+\int_0^t \mathbf{K}(t-\tau)\ast \tilde F(v,u)(\tau)d\tau,
\end{align*}
we obtain that \begin{align*}
\sup_{t\in[0,T]}\|u(t)\|_{L^1}&\lesssim \|u_0\|_{L^1}+\int_0^t (t-\tau)^{-\frac{1}{2}}\|\tilde F(v,u)(\tau)\|_{L^1}d\tau\lesssim M_1+T^\gamma\|\tilde F(v,u)\|_{X_T^{\frac{1}{2}-\gamma,1}}\nonumber\\
&\overset{\eqref{tilF}}\lesssim (1+c_1T^\gamma)M_1\lesssim M_1,\\
\sup_{t\in[0,T]}t^\frac{1}{2}\|u(t)\|_{L^\infty}&\lesssim \sup_{t\in[0,T]}t^\frac{1}{2}\|\mathbf{K}(t)\|_{L^\infty}\|u_0\|_{L^1}+\sup_{t\in[0,T]}t^\frac{1}{2}\int_0^t\|\tilde F(v,u)(\tau)\|_{L^\infty}d\tau\\
&\lesssim \|u_0\|_{L^1}+t^{\frac{1}{2}+\gamma} \|\tilde F(v,u)\|_{X_T^{1-\gamma,\infty}}\overset{\eqref{tilF}}\lesssim M_1.
\end{align*}
Combining this with \eqref{re11}, we get 
\begin{align}\label{re13}
&\|(v(t)-1,u(t))\|_{L^1\cap BV}+\sqrt{t+1}\|(v(t)-1,u(t))\|_{L^\infty}+\sqrt{t}\|u_x(t)\|_{L^\infty}\lesssim  t^{-\frac{1}{2}}M_1,\ \ \ \ \forall t\in[0,T].
\end{align}
We can take $\varepsilon_0$ small enough such that 
$$
\|(v({T}/{2})-1,u({T}/{2}))\|_{L^1\cap BV}\lesssim T^{-\frac{1}{2}}M_1\lesssim T^{-\frac{1}{2}}\varepsilon_0\ll 1,
$$
satisfies the condition \eqref{theircon}. By \cite[Theorem 1.4]{LiuCAPM}, the Cauchy problem of system \eqref{inns} with initial data $(\tilde v_0,\tilde u_0)=(v(T/2),u(T/2))$ has a unique global solution $(\tilde v,\tilde u)$ such that $(\tilde v,\tilde u)|_{t=0}=(v(T/2),u(T/2))$, and 
   \begin{align}\label{re12}
&\|(\tilde v(t)-1,\tilde u(t))\|_{L^1\cap BV}+\sqrt{t+1}\|(\tilde v(t)-1,\tilde u(t))\|_{L^\infty}+\sqrt{t}\|\tilde u_x(t)\|_{L^\infty}\lesssim M_1,\ \ \ \ \forall t>0.
\end{align}
Due to uniqueness, we have $(\tilde v,\tilde u)(t)=(v,u)(t+T/2)$. By this way we extend the solution $(v,u)$ globally in time. By \eqref{re12}, we have 
\begin{align*}
&\|(v(t)-1,u(t))\|_{L^1\cap BV}+\sqrt{t+1}\|( v(t)-1, u(t))\|_{L^\infty}+\sqrt{t-T/2}\| u_x(t)\|_{L^\infty}\lesssim M_1,\ \ \ \ \forall t>T/2.
\end{align*}
Combining this with \eqref{re13}, we obtain  \eqref{esthm1}. 
This completes the proof.
\end{proof}
\begin{remark}
From the above proof, we can see that it is enough to put $(1+\frac{1}{\sqrt{t}})M_1$ in the right hand side of \eqref{esthm1}.
\end{remark}
\section{Well-posedness for the full Navier–Stokes system }
	We define the following norms and spaces
	\begin{align*}
		&\|u\|_{Z_T^1}=\sum_{\star\in\{L_{T,x}^2,X_T^{\frac{1}{2},2},X_T^{\frac{3}{4},\infty},X_T^{\frac{1}{2}-\gamma,1},BV_T^{1-\gamma}\}}\|u_x\|_{\star},\\
		&	\|\theta\|_{Z_T^2}=\sum_{\star\in\{L_{T,x}^2,X_T^{\frac{1}{2},2},X_T^{\frac{1}{2}-\gamma,1}\}}\|\theta\|_{\star}+\sum_{\star\in\{L_{T,x}^{\frac{6}{5}},X_T^{\frac{5}{6},\frac{6}{5}},X_T^{1-\gamma,1}\}}\|\theta_x\|_{\star}
		,\\
		&\mathcal{G}_{T,\eta}=\{(u,\theta)|(u(0,\cdot),\theta(0,\cdot))=(u_0,\theta_0), \|(u,\theta)\|_{Z_T}=\|u\|_{Z_T^1}+\|\theta-1\|_{Z_T^2}\leq\eta\}.
	\end{align*}
We want to prove the following main theorem in this section. 
	\begin{theorem}	\label{thmfull} There exist $\varepsilon_1,T\in(0,1)$ such that, for  any initial data $(v_0,u_0,\theta_0)$  satisfying
		\begin{equation}
			M_2:=\|v_0-1\|_{L^1 \cap BV}+\|u_0\|_{W^{2\gamma,1}\cap L^2}+\|\theta_0-1\|_{\dot W^{-\frac{2}{3},\frac{6}{5}}\cap \dot W^{2\gamma-1}}\leq\eps_1,
		\end{equation}
		the system \eqref{cpns} admits a unique solution $(v,u,\theta)$ satisfying
		\begin{equation}\label{thm4.1}
			\|v-1\|_{Y_T}+\|u\|_{Z_T^1}+\|\theta-1\|_{Z_T^2}\lesssim M_2.
		\end{equation}
	\end{theorem}
	\begin{proof} As in the proof of Theorem \ref{thmisen}, we construct a solution $(v,u,\theta)$ of  the system \eqref{cpns}  as a fixed point of the following  map.
		For $(w,\vartheta)\in\mathcal{G}_{T,\eta}$,	we define the  map
		\begin{equation*}
			(u,\theta)=\mathcal{T}(w,\vartheta),
		\end{equation*}
		where  $(u,\theta)$ solves the following equation
		\begin{equation}\label{fnsfml}
			\begin{aligned}
				&	u_t-\mu u_{xx}+(P(v,\theta))_x=\mu\left(\left(\frac{1}{v}-1\right)\partial_x w\right)_x,\ \ \ \ 
				u(0,x)=u_0(x),
				\\
				&		\theta_{t}-\frac{\kappa}{\mathbf{c}}\theta_{xx}=-\frac{P(v,\vartheta)}{\mathbf{c}} w_{x}+\frac{\mu}{\mathbf{c} v}\left(w_{x}\right)^{2}+\frac{\kappa}{\mathbf{c}}\left(\left(\frac{1}{v}-1\right)\partial_x \vartheta\right)_x,\ \ \ \theta(0,x)=\theta_0(x),
			\end{aligned}
		\end{equation}
		where $P(v,\theta)=\frac{K\theta}{v}$ and $v$ is fixed by
		\begin{equation}\label{fnsv}
			v(t,x)=v_0(x)+\int_0^tw_x(s,x)ds.
		\end{equation}
		We can write the formula for $u$ and $\theta$ by \eqref{fnsfml}, 
		\begin{align*}
			&\theta(t)=\mathcal{T}_2(w,\vartheta)(t)=\mathbf{K}(t)\ast \theta_0+\int_0^t\mathbf{K}(t-s)\ast R(s)ds +\int_0^t\mathbf{K}(t-s)\ast \partial_xF(s)ds,\\
			&		u(t)=\mathcal{T}_1(w,\theta)(t)=\mathbf{K}(t)\ast u_0+\int_0^t\mathbf{K}(t-s)\ast \partial_xG(s)ds,
		\end{align*}
		where $$R=-\frac{K}{\mathbf{c}}\frac{\vartheta  w_x}{v}+\frac{\mu}{\mathbf{c} v}w_{x}^2,\ \ \ \  F=\frac{\kappa}{\mathbf{c}}(\frac{1}{v}-1)\vartheta_x,\ \ \ \  G=-K(\frac{\theta}{v}-1)+\mu(\frac{1}{v}-1)w_x.$$
		We will prove that for any $c>0$, there exists $\eps_1>0$ such that for any $M_2<\eps_1$,
		\begin{equation}\label{kpnorm}
			\|\mathcal{T}(w,\vartheta)\|_{Z_T}\leq C_0(1+cT^{\gamma})M_2,\quad \forall (w,\vartheta)\in \mathcal{G}_{T,cM_2},
		\end{equation}
		\begin{equation}\label{commap}
			\|\mathcal{T}(w_1,\vartheta_1)-\mathcal{T}(w_2,\vartheta_2)\|_{{Z}_T}\leq C_0cM_2(1+T^{\gamma})\|(w_1-w_2,{\vartheta}_1-\vartheta_2)\|_{{Z}_T},\ \forall (w_i,\vartheta_i)\in\mathcal{G}_{T,cM_2},\  i=1,2,
		\end{equation}
		for any $T\in(0,1)$, and $C_0$ independent of $c$.\\
		First we prove \eqref{kpnorm}. Similar as in the proof of Theorem \ref{thmisen}, we can prove that there exists constant $M>0$ such that 
		\begin{equation}\label{fnsv1}
			\|v-1\|_{Y_T}\leq C_1(M_2+T^\gamma\|w\|_{Z_T^1})\leq C_1(1+cT^\gamma)M_2.
		\end{equation}
		Then we obtain $\|v-1\|_{Y_T}\leq \frac{1}{2}$ by taking $\eps_1<\frac{1}{(10+C_1+c)^5}$ such that $C_1(1+c)\eps_1\leq \frac{1}{2}$.\\
		For $\theta$, we consider the equation of $\theta-1$ and we have that, by Lemma \ref{lemma},
		\begin{equation}\label{fnstheta1}
		\begin{aligned}
			&\|\theta-1\|_{L_{T,x}^2\cap X_T^{\frac{1}{2},2}\cap X_T^{\frac{1}{2}-\gamma,1}}+\|\partial_x\theta\|_{L_{T.x}^\frac{6}{5}\cap X_T^{\frac{5}{6},\frac{6}{5}}\cap X_T^{1-\gamma,1}}\\
			&\lesssim \|\theta_0-1\|_{\dot W^{-\frac{2}{3},\frac{6}{5}}\cap \dot W^{2\gamma-1,1} }+T^{\frac{1}{4}}\|R\|_{L_{T,x}^1\cap X_T^{1,1}}+\|F\|_{L^{\frac{6}{5}}_{T,x}\cap X_T^{\frac{5}{6},\frac{6}{5}}\cap X_T^{1-\gamma,1}}.
		\end{aligned}
		\end{equation}
			For $R$, we can write as
		\begin{equation*}
			R=\frac{\mu}{\mathbf{c}v}w_x^2-\frac{Kw_x}{\mathbf{c}v}-\frac{K}{\mathbf{c}}\frac{\vartheta-1}{v}w_x.
		\end{equation*}
	By H\"{o}lder's inequality, it is easy to check that 
		\begin{align*}
			&\|R\|_{L_{T,x}^1}\lesssim \|w\|_{L_{T,x}^2}^2+\|w_x\|_{L_{T,x}^1}+\|\vartheta-1\|_{L_{T,x}^2}\|w_x\|_{L_{T,x}^2}\lesssim c(T^{\gamma+\frac{1}{2}}+cM_2)M_2,\\
		&	\sup_{t\in[0,T]}t\|R(t)\|_{L^1}\lesssim\|w_x\|_{X_T^{\frac{1}{2},2}}^2+T^{\frac{1}{2}+\gamma}\|w_x\|_{X_T^{\frac{1}{2}-\gamma,1}}+\|\vartheta-1\|_{X_T^{\frac{1}{2},2}}\|w_x\|_{X_T^{\frac{1}{2},2}}\lesssim c(T^{\frac{1}{2}+\gamma}+cM_2)M_2.
		\end{align*}
		\begin{equation*}
			\begin{aligned}
				&\sup_{s<t<T}s^{1+\alpha}\frac{\|R(t)-R(s)\|_{L^1}}{(t-s)^\alpha}\\
				\lesssim& \sup_{s<t<T}s^{1+\alpha}\frac{\|w_x(t)-w_x(s)\|_{L^2}(\|w_x(t)\|_{L^2}+\|w_x(s)\|_{L^2})}{(t-s)^\alpha}+\sup_{s<t<T}s^{1+\alpha}\|w_x(s)\|_{L^2}^2\frac{\|v(t)-v(s)\|_{L^\infty}}{(t-s)^{\alpha}}\\
				&+\sup_{s<t<T}s^{1+\alpha}\frac{\|w_x(t)-w_x(s)\|_{L^1}}{(t-s)^\alpha}+\sup_{s<t<T}s^{1+\alpha}\|w_x(s)\|_{L^\infty}\frac{\|v(t)-v(s)\|_{L^1}}{(t-s)^\alpha}\\
				&+\sup_{s<t<T}s^{1+\alpha}\|w_x(t)\|_{L^2}\frac{\|\vartheta(t)-\vartheta(s)\|_{L^2}}{(t-s)^{\alpha}}+\sup_{s<t<T}s^{1+\alpha}\|\vartheta(s)-1\|_{L^2}\frac{\|w_x(t)-w_x(s)\|_{L^2}}{(t-s)^{\alpha}}\\
				&+\sup_{s<t<T}s^{1+\alpha}\|w_x(s)\|_{L^2}\|\vartheta(s)-1\|_{L^2}\frac{\|v(t)-v(s)\|_{L^\infty}}{(t-s)^{\alpha}}\\
				\lesssim& c(T^{\frac{1}{2}+\gamma}+cM_2)M_2.
			\end{aligned}
		\end{equation*}
The above estimates yield
		\begin{equation}\label{fnsrem1}
			\begin{aligned}
				\|R\|_{L_{T,x}^1\cap X_T^{1,1}}\lesssim c(T^{\frac{1}{2}+\gamma}+cM_2)M_2.
			\end{aligned}
		\end{equation}
		Applying Lemma \ref{lemproduct} with $g=\frac{1}{v}-1$ and $h=\vartheta_x$ to obatin 
		\begin{align*}
			\|F\|_{X^{\frac{5}{6},\frac{6}{5}}_T\cap X_T^{1-\gamma,1}}&\lesssim \left\|\frac{1}{v}-1\right\|_{X^{0,\infty}_T}\|\vartheta_x\|_{X^{\frac{5}{6},\frac{6}{5}}_T\cap X_T^{1-\gamma,1}}\lesssim \|v-1\|_{X^{0,\infty}_T}\|\vartheta_x\|_{X^{\frac{5}{6},\frac{6}{5}}_T\cap X_T^{1-\gamma,1}}\\
			&\lesssim \|v-1\|_{Y_T}\|\vartheta-1\|_{Z_T^2}.
		\end{align*}
		Moreover, by Holder's inequality,
		\begin{equation}\label{fnsthetaf1}
			\|F\|_{L^\frac{6}{5}_{T,x}}\lesssim \|v-1\|_{L_{T,x}^\infty}\|\vartheta_x\|_{L_{T,x}^{\frac{6}{5}}}\lesssim \|v-1\|_{Y_T}\|\vartheta-1\|_{Z_T^2}.
		\end{equation}
		Hence,
		\begin{equation}\label{fnsthetaf}
	\|F\|_{L^{\frac{6}{5}}_{T,x}\cap X_T^{\frac{5}{6},\frac{6}{5}}\cap X^{1-\gamma,1}_T}\lesssim \|v-1\|_{Y_T}\|\vartheta-1\|_{Z_T^2}\overset{\eqref{fnsv1}}\lesssim c(1+cT^\gamma)M_2^2.
		\end{equation}
		So we substitute \eqref{fnsrem1},  \eqref{fnsthetaf} into \eqref{fnstheta1} to have the estimate
		\begin{align}\label{fnstheta}
			\|\theta-1\|_{Z^2_T}
			\lesssim \|\theta_0-1\|_{\dot W^{-\frac{2}{3},\frac{6}{5}}\cap \dot W^{2\gamma-1,1}}+(1+cT^\gamma)M_2\lesssim (1+cT^\gamma)M_2.
		\end{align}
		Then we estimate $u$,  apply Lemma \ref{lemma} again to $u$ to get that
		\begin{equation}\label{fnsu1}
		\|\partial_xu\|_{\star}\lesssim \|u_0\|_{L^2\cap \dot W^{2\gamma,1}}+\|G\|_{\star},\ \ \ \ \ \ \ \star\in { \{L_{T,x}^2,X_T^{\frac{1}{2},2},X_T^{\frac{3}{4},\infty},X_T^{\frac{1}{2}-\gamma,1},BV^{1-\gamma}_T\}}.
		\end{equation}
Recall that 
		\begin{equation*}
			G=-K(\frac{\theta}{v}-1)+\mu(\frac{1}{v}-1)w_x=G_1+G_2.
		\end{equation*}
	For $G_1$, by Lemma \ref{lemproduct} we have, for $\star\in { \{L_{T,x}^2,X_T^{\frac{1}{2},2},X_T^{\frac{3}{4},\infty},X_T^{\frac{1}{2}-\gamma,1},BV^{1-\gamma}_T\}}$,
		\begin{equation*}
			\begin{aligned}
				\|G_1\|_{\star}&\lesssim \left\|\frac{\theta-1}{v}\right\|_{\star}+ \left\|\frac{v-1}{v}\right\|_{\star}\lesssim (\|\theta-1\|_{\star}+\|v-1\|_{\star})(1+\|v-1\|_{Y_T}).
			\end{aligned}
		\end{equation*}
	From	the interpolation $\|\theta-1\|_{X_T^{\frac{3}{4},\infty}}\lesssim \|\theta-1\|_{X_T^{\frac{1}{2},2}}+\|\theta_x\|_{X_T^{\frac{5}{6},\frac{6}{5}}}\lesssim \|\theta-1\|_{Z_T^2}$, and the fact that $\|\theta-1\|_{BV^{1-\gamma}_T}\lesssim \|\theta_x\|_{X_T^{1-\gamma,1}}\lesssim \|\theta-1\|_{Z_T^2}$, we obtain that $\|\theta-1\|_{\star}\lesssim \|\theta-1\|_{Z_T^2}$. Moreover, it is easy to check that $\|v-1\|_{\star}\lesssim \|v-1\|_{Y_T}$, for any $\star\in { \{L_{T,x}^2,X_T^{\frac{1}{2},2},X_T^{\frac{3}{4},\infty},X_T^{\frac{1}{2}-\gamma,1},BV^{1-\gamma}_T\}}$. Combining this with \eqref{fnsv1} and \eqref{fnstheta} to obtain that 
	\begin{align}\label{G1}
	   				\sum_{\star\in { \{L_{T,x}^2,X_T^{\frac{1}{2},2},X_T^{\frac{3}{4},\infty},X_T^{\frac{1}{2}-\gamma,1},BV^{1-\gamma}_T\}}}\|G_1\|_{\star}
	   				&\lesssim (1+c T^\gamma)M_2.
	\end{align}
		For $G_2$, apply Lemma \ref{lemproduct} again to obtain 
		\begin{align*}
			&\|G_2\|_{L_{T,x}^2}\lesssim \|v-1\|_{L_{T,x}^\infty}\|w_x\|_{L_{T,x}^2}\lesssim c(1+c T^\gamma)M_2^2,
			\\
			&\|G_2\|_{\star}\lesssim \|v-1\|_{Y_T}\|w_x\|_{\star}\lesssim c(1+c T^\gamma)M_2^2,\ \ \ \  \star \in\{X_T^{\frac{1}{2},2}, X_T^{\frac{3}{4},\infty},X_T^{\frac{1}{2}-\gamma,1},BV^{1-\gamma}_T\}.
		\end{align*}
		Hence,
		\begin{equation}\label{fnsg}
			\begin{aligned}
				\sum_{\star\in \{L_{T,x}^2,X_T^{\frac{1}{2},2},X_T^{\frac{3}{4},\infty},X_T^{\frac{1}{2}-\gamma,1},BV^{1-\gamma}_T\}}\|G_2\|_{\star}
				\lesssim (1+c T^\gamma)M_2.			\end{aligned}
		\end{equation}
		Substitute \eqref{G1} and \eqref{fnsg} into \eqref{fnsu1}, one has 
		\begin{align}\label{fnsu}
		   \|u\|_{Z_T^1}\lesssim  \|u_0\|_{L^2\cap \dot W^{2\gamma,1}}+(1+c T^\gamma)M_2\lesssim (1+c T^\gamma)M_2.
		\end{align}
Then we obtain \eqref{kpnorm} as a result of \eqref{fnstheta} and \eqref{fnsu}.\\
		
		Then we prove the contraction property \eqref{commap}. For $(w_i,\vartheta_i)\in \mathcal{G}_{T,cM_2}$, denote $(u_i,\theta_i)=\mathcal{T}(w_i,\vartheta_i), i=1,2$, and 
		$$
		v_i(t)=v_0+\int_0^t \partial_xw_i(\tau)d\tau.
		$$
	For simplicity, we denote $\bar{f}=f_1-f_2$ for $f\in\{v,u,\theta,w,\vartheta\}$, and 
	$$
	\bar R=R(w_1,\vartheta_1)-R(w_2,\vartheta_2),\ \ \bar F=F(v_1,\vartheta_1)-F(v_2,\vartheta_2),\ \ \ \bar G=G(v_1,w_1,\theta_1)-G(v_2,w_2,\theta_2).
	$$
	We can write the equations
		\begin{equation*}
			\begin{aligned}
				&\partial_t\bar{u}-\frac{\mu}{\mathbf{c}}\bar{u}_{xx}=\partial_x\bar{G},\\
				&\partial_t\bar{\theta}-\frac{\kappa}{\mathbf{c}}\bar\theta_{xx}=\bar{R}+\partial_x\bar{F}.
			\end{aligned}
		\end{equation*}
		By the equations above we can write the formula of $\bar{u}$ and $\bar{\theta}$ as follows
		\begin{align*}
&	\bar{u}(t)=\int_0^t\mathbf{K}(t-s)\ast\bar{G}(s)ds,\\
&\bar{\theta}(t)=\int_0^t\mathbf{K}(t-s)\ast\bar{R}(s)ds+\int_0^t\mathbf{K}(t-s)\ast\bar{F}(s)ds.
		\end{align*}
		Following the proof of \eqref{fnsv1}, we have $\|v_1-1\|_{Y_T}+\|v_2-1\|_{Y_T}\lesssim (1+cT^\gamma)M_2$, and 
		\begin{equation*}
		\|\bar{v}\|_{{Y}_T}\lesssim T^{\gamma}\|\bar{w}\|_{{Z}_T^1}.
		\end{equation*}
	Thanks to Lemma \ref{lemma}, we have
		\begin{align}
			\|\bar \theta\|_{Z_T^2}&=\|\bar{\theta}\|_{L_{T,x}^2\cap X_T^{\frac{1}{2},2}\cap X_T^{\frac{1}{2}-\gamma,1} }+\|\partial_x\bar{\theta}\|_{L_{T,x}^{\frac{6}{5}}\cap X_T^{\frac{5}{6},\frac{6}{5}},X_T^{1-\gamma,1}}\nonumber\\
			&\lesssim T^{\frac{1}{4}}\|\bar{R}\|_{L^1_{T,x}\cap X_T^{1,1}}+\|\bar{F}\|_{L^\frac{6}{5}\cap X_T^{\frac{5}{6},\frac{6}{5}}\cap X_T^{1-\gamma,1}}.\label{fnsbtheta3}
		\end{align}
		\begin{equation}\label{fnsbu4}
		\|\bar u\|_{Z_T^1}=	\sum_{\star\in \{L_{T,x}^2,X_T^{\frac{1}{2},2},X_T^{\frac{3}{4},\infty},X_T^{\frac{1}{2}-\gamma,1},BV_T^{1-\gamma}\}}\|\partial_x\bar{u}\|_{\star}\leq\sum_{\star\in \{L_{T,x}^2,X_T^{\frac{1}{2},2},X_T^{\frac{3}{4},\infty},X_T^{\frac{1}{2}-\gamma,1},BV_T^{1-\gamma}\}}\|\bar{G}\|_{\star}.
		\end{equation}
With a slight abuse of notation, we drop some constants and  write the formula of $\bar{R}$, $\bar{F}$ and $\bar{G}$ as follows
		\begin{align*}
			&\bar{R}=\frac{\bar{\vartheta}}{v_1}w_{1x}+\frac{\vartheta_2-1}{v_1v_2}\bar{v}w_{1x}+\frac{\vartheta_2-1}{v_2}\bar{w}_x+\frac{\bar{v}}{v_1v_2}w_{1x}+\frac{\bar{w}_x}{v_2}+\frac{\bar{v}}{v_1v_2}w_{1x}^2+\frac{\bar{w}_x(w_{1x}+w_{2x})}{v_2},\\
			&\bar{F}=\frac{\bar{v}}{v_1v_2}\vartheta_{1x}+(\frac{1}{v_2}-1)\bar{\vartheta}_x,\\
			&\bar{G}=\frac{\bar{v}}{v_1v_2}(\theta_1-1)+\frac{1}{v_2}\bar{\theta}+\frac{\bar{v}}{v_1v_2}+\frac{\bar{v}}{v_1v_2}w_{1x}+(\frac{1}{v_2}-1)\bar{w}_x.
		\end{align*}
		
		So we can do as \eqref{fnsrem1}, \eqref{fnsthetaf1}, \eqref{fnsthetaf} and \eqref{fnsg} to get the following estimates, for any $\star\in\{L_{T,x}^2,X_T^{1,1}\}$,
		\begin{align*}
    			\|\bar{R}\|_{\star}\lesssim &\|\bar{\vartheta}\|_{Z_T^2}\|\frac{1}{v_1}\|_{Y_T}\|w_1\|_{Z_T^1}+\|\vartheta_2\|_{Z_T^2}\|\frac{\bar{v}}{v_1v_2}\|_{Y_T}\|w_1\|_{Z_T^1}+\|\vartheta\|_{Z_T^2}\|\frac{1}{v_2}\|_{Y_T}\|\bar{w}\|_{Z_T^1}+T^{\frac{1}{2}+\gamma}\|\frac{\bar{v}}{v_1v_2}\|_{Y_T}\|w_1\|_{Z_T^1}\\	&+T^{\frac{1}{2}+\gamma}\|\frac{1}{v_2}\|_{Y_T}\|\bar{w}_x\|_{Z_T^1}+\|\frac{\bar{v}}{v_1v_2}\|_{Y_T}\|w_1\|_{Z_T^1}^2+\|\bar{w}\|_{Z_T^1}\|\frac{1}{v_2}\|_{Y_T}(\|w_1\|_{Z_T^1}+\|w_2\|_{Z_T^1})\\
    			\lesssim &cM_2(1+T^{\frac{1}{2}+\gamma})(\|\bar{w}\|_{Z_T^1}+\|\bar{\vartheta}\|_{Z_T^2}).
    			\end{align*}
    			For any $\star\in\{L_{T,x}^{\frac{6}{5}},X_T^{\frac{5}{6},\frac
    			{6}{5}},X_T^{1-\gamma,1}\}$, by Lemma \ref{lemproduct},
    			\begin{align*}
			\|\bar{F}\|_{\star}\lesssim \|\frac{\bar{v}}{v_1v_2}\|_{Y_T}\|\vartheta\|_{\star}+\|\frac{v_2-1}{v_2}\|_{Y_T}\|\bar\vartheta_x\|_{\star}\lesssim M_2(1+T^{\gamma})(\|\bar{w}\|_{Z_T^2}+\|\bar{\vartheta}\|_{Z_T^2}).
			\end{align*}
			For any $\star\in\{L_{T,x}^2,X_T^{\frac{1}{2},2},X_T^{\frac{3}{4},\infty},X_T^{\frac{1}{2}-\gamma,1},BV_T^{1-\gamma}\}$, we have
			\begin{align*}
\|\bar G\|_{\star}&\lesssim \left\|\frac{\bar{v}}{v_1v_2}\right\|_{Y_T}(\|\theta_1-1\|_{\star}+\|w_{1x}\|_{\star})+\left\|\frac{1}{v_2}\right\|_{Y_T}\|\bar \theta\|_{\star}+\left\|\frac{1}{v_1v_2}\right\|_{Y_T}\|\bar v\|_{\star}+\left\|\frac{1}{v_2}-1\right\|_{Y_T}\|\bar w_x\|_\star\\
&\lesssim \|\bar \theta\|_{\star }+M_2(1+cT^\gamma)\|\bar w_x\|_{\star}.
		\end{align*}
		So we get from  \eqref{fnsbtheta3} and \eqref{fnsbu4} that
		\begin{equation}\label{fnsbtheta1}
			\|\bar \theta\|_{Z_T^2}\lesssim  cM_2(1+T^{\gamma})(\|\bar{w}\|_{Z_T^1}+\|\bar{\vartheta}\|_{Z_T^2}),	
			\end{equation}
		\begin{equation}\label{fnsbu1}
		\|\bar u\|_{Z_T^1}\lesssim  cM_2(1+T^{\gamma})(\|\bar{w}\|_{Z_T^1}+\|\bar{\vartheta}\|_{Z_T^2}).
		\end{equation}
		 So we combine \eqref{fnsbtheta1}, \eqref{fnsbu1}, and then we get \eqref{commap}. Following the proof of Theorem \ref{thmisen}, we apply the Banach fixed point theorem and use \eqref{kpnorm} together with \eqref{commap}, we take $\eps_1$ small enough, and $c_2>1$ such that for any $M_2<\eps_1$, $C_0c_2M_2(1+T^{\gamma})<\frac{1}{2}$, then we finish the proof of existence and uniqueness of solution in $\mathcal{G}_{T,c_2M_2}$. 
	\end{proof}\vspace{0.3cm}\\
	\begin{proof}[Proof of Theorem \ref{thm2}] 
	Let $(v,u,\theta)$ be the unique solution in Theorem \ref{thmfull}. By \eqref{thm4.1} we have 
		\begin{equation}\label{firsthalf}
			\begin{aligned}
		\|v(t)-1\|_{L^1\cap BV\cap L^\infty}+t^{\frac{1}{2}}\|u(t)\|_{ BV}+t^{\frac{3}{4}}\|u_x(t)\|_{L^\infty}+t^{1-\gamma}\|\theta(t)-1\|_{L^1\cap BV\cap L^\infty}(t) \lesssim M_2.
			\end{aligned}
		\end{equation}
Moreover, by the formula 
		\begin{equation*}
			u(t)=\mathbf{K}(t)\ast u_0+\int_0^t\mathbf{K}(t-s)\ast \partial_xG(s)ds.
		\end{equation*}
		It is easy to check that 
		\begin{equation}\label{fuluL1}
	\begin{aligned}
		  \sup_{t\in[0,T]}\|u(t)\|_{L^1}&\lesssim \|u_0\|_{L^1}+\int_0^t(t-\tau)^{-\frac{1}{2}}\|G(\tau)\|_{L^1}d\tau\lesssim M_2+T^\gamma \|G\|_{X^{\frac{1}{2}-\gamma,1}_T}\\
		   &\lesssim (1+T^\gamma)M_2\lesssim M_2,\\
		  \sup_{t\in[0,T]}t^\frac{1}{2}\|u(t)\|_{L^\infty}&\lesssim \sup_{t\in[0,T]}t^\frac{1}{2}\|\mathbf{K}(t)\|_{L^\infty}\|u_0\|_{L^1}+\sup_{t\in[0,T]}t^\frac{1}{2}\int_0^t\| G(\tau)\|_{L^\infty}d\tau\\
&\lesssim \|u_0\|_{L^1}+t^{\frac{1}{4}} \|G\|_{X_T^{\frac{3}{4},\infty}}\overset{\eqref{tilF}}\lesssim M_2.
		\end{aligned}
		\end{equation}
		For $\|\theta_x\|_{X_T^{\frac{3}{2}-\gamma,\infty}}$, we have the formula
		\begin{equation*}
		\theta_x(t)=\partial_xK(t)\ast (\theta_0-1)+\int_0^t\partial_xK(t-s)\ast R(s)ds+\int_0^t\partial_{xx}K(t-s)\ast F(s)ds:=\theta_L+\theta_R+\theta_N.
		\end{equation*}
		It is easy to see that 
		\begin{equation*}
		\|\theta_L\|_{X_T^{\frac{3}{2}-\gamma,\infty}}\lesssim \|\theta_0-1\|_{\dot W^{2\gamma-1,1}}.
		\end{equation*}
		By Lemma \ref{rem}, we obtain
		\begin{equation*}
		\|\theta_R\|_{X_T^{\frac{3}{2}-\gamma,\infty}}\lesssim T^{\frac{1}{2}-\gamma}(\|R\|_{L_{T,x}^1}+\|R\|_{X_T^{\frac{3}{2},\infty}}).
		\end{equation*}
		By Lemma \ref{mainlem}, we have
		\begin{equation*}
		\|\theta_N\|_{X_T^{\frac{3}{2}-\gamma,\infty}}\lesssim \|F\|_{X_T^{\frac{3}{2}-\gamma,\infty}}+\|F\|_{X_T^{1-\gamma,1}}. 
		\end{equation*}
		And we have 
		\begin{equation*}
		\|R\|_{X_T^{\frac{3}{2}-\gamma,\infty}}\lesssim (1+\|v-1\|_{Y_T})(\|u\|_{Z_T^1}+\|\theta-1\|_{Z_T^2})\|u\|_{Z_T^1}\lesssim M_1^2(1+T^\gamma M_1)
		\end{equation*}
		and 
		\begin{equation*}
		\|F\|_{X_T^{\frac{3}{2}-\gamma,\infty}}\lesssim \|v-1\|_{Y_T}(1+\|v-1\|_{Y_T})\|\theta_x\|_{X_T^{\frac{3}{2}-\gamma,\infty}}\lesssim M_1 \|\theta_x\|_{X_T^{\frac{3}{2}-\gamma,\infty}}.
		\end{equation*}
		Combining the estimates above with \eqref{fnsrem1} and \eqref{fnsthetaf} to obtain that
		\begin{equation*}
		\|\theta_x\|_{X_T^{\frac{3}{2}-\gamma,\infty}}\leq C(\|\theta_0-1\|_{\dot W^{2\gamma-1,1}}+T^{\frac{1}{2}-\gamma}M_1+M_1\|\theta_x\|_{X_T^{\frac{3}{2}-\gamma,\infty}}).
		\end{equation*}
		So we can choose $\varepsilon_1$ small such that $CM_1\leq C\varepsilon_1\leq\frac{1}{2}$, then we have 
		\begin{equation}\label{highord}
		\|\theta_x\|_{X_T^{\frac{3}{2}-\gamma,\infty}}\leq 4CM_1.
		\end{equation}
Combining \eqref{firsthalf}, \eqref{fuluL1} and \eqref{highord} to obtain that for $U(t)=(v(t)-1,u(t),\theta(t)-1)$,
\begin{align}\label{es0T}
      &	\|U(t)\|_{L^1\cap BV}+\sqrt{t+1}\|U(t)\|_{L^\infty}+\sqrt{t}\|(u_x(t),\theta_x(t))\|_{L^\infty}\lesssim \frac{1}{t}M_2,\ \ \ \  \forall t\in[0,T].
\end{align}
		We can further take $\eps_1$ small enough such that
		\begin{equation*}
		\|v(T/2)-1\|_{L^1\cap BV}+\|u(T/2)\|_{L^1\cap BV}+\|\theta(T/2)-1\|_{L^1\cap BV}\lesssim T^{\gamma-1}M_2\leq T^{\gamma-1}\eps_1\ll 1.
		\end{equation*}
	By \cite[Theorem 6.1]{Wang}, there exists a unique global solution $(\tilde{v},\tilde{u},\tilde{\theta})(t)$ for Cauthy problem with initial data $(v,u,\theta)(\frac{T}{2})$ with 
		\begin{equation}\label{continuation}
		\|(\tilde{v}-1,\tilde{u},\tilde{\theta}-1)(t)\|_{L^1\cap BV}+\sqrt{t+1}\|(\tilde{v}-1,\tilde{u},\tilde{\theta}-1)(t)\|_{L^\infty}+\sqrt{t}\|(\tilde{u}_x,\tilde\theta_x)(t)\|_{L^\infty}\lesssim M_2.
		\end{equation}
		And by uniqueness, the $(\tilde{v},\tilde{u},\tilde{\theta})$ is the extension of $(v,u,\theta)$ on $(\frac{T}{2},\infty)$. We obtain \eqref{main2} by \eqref{es0T} and  \eqref{continuation}, and the proof is done. 
	\end{proof}\\

	\textbf{Acknowledgements:} This research is funded by Vietnam National University Ho Chi Minh City (VNU-HCM) under grant number T2022-18-01.
	Q.H.N.  is supported by the Academy of Mathematics and Systems Science, Chinese Academy of Sciences startup fund, and the National Natural Science Foundation of China (No. 12050410257 and No. 12288201) and  the National Key R$\&$D Program of China under grant 2021YFA1000800. 

\end{document}